\documentclass[10pt,leqno,twoside]{amsart}

\usepackage{enumerate}
\usepackage{xr-hyper}
\usepackage{subfigure}
\makeatletter

\usepackage{amssymb}
\usepackage[english]{babel}
\usepackage{amssymb,amsthm,amsmath,eucal,mathrsfs}
\usepackage{bm}

\usepackage{amsmath}
\usepackage{amsfonts}
\usepackage{amsmath}
\usepackage{amssymb}
\usepackage{graphicx}
\usepackage{lscape}
\usepackage{amstext}
\usepackage{amsthm}
\usepackage{color}
\usepackage{float}
\usepackage{mathrsfs}
\usepackage{epsfig}
\usepackage{url}
\usepackage{fancyhdr}
\usepackage{pspicture}
\usepackage{graphicx}
\makeatletter
\newcommand*{\addFileDependency}[1]{
\typeout{(#1)}
\@addtofilelist{#1}
\IfFileExists{#1}{}{\typeout{No file #1.}}
}\makeatother

\usepackage{cite}

\usepackage{amsmath,verbatim}

\usepackage{amsthm}

\usepackage{amssymb}

\usepackage{amsfonts}

\usepackage{dsfont}

\usepackage{hyperref}

\newtheorem{theorem}{Theorem}[section]

\newtheorem{lemma}[theorem]{Lemma}

\newtheorem{proposition}[theorem]{Proposition}

\newtheorem{corollary}[theorem]{Corollary}

\theoremstyle{definition}

\newtheorem{remark}[theorem]{Remark}

\numberwithin{equation}{section}

\newcommand\restr[2]{{
  \left.\kern-\nulldelimiterspace 
  #1 
  \vphantom{\big|} 
  \right|_{#2} 
  }}

\usepackage{eucal}

\title{Characterization of the Eigenvalues and Eigenfunctions of the Helmholtz Newtonian operator $\boldsymbol{N}^{\boldsymbol{k}}$}

\author[Zhe Wang, Ahcene  Ghandriche and Jijun Liu]
{Zhe Wang$^*$ Ahcene Ghandriche$^{**}$ and Jijun Liu$^{\ddag}$}

\thanks{$^{*}$School of Mathematics, Southeast University, Nanjing, 210096, P.R.China. Nanjing Center for Applied Mathematics, Nanjing, 211135, P.R.China. Email: zhewang23@seu.edu.cn.}

\thanks{$^{**}$Nanjing Center for Applied Mathematics, Nanjing, 211135, P.R.China. Email: gh.hsen@njcam.org.cn.}

\thanks{$^{\ddag}$School of Mathematics, Southeast University, Nanjing, 210096, P.R.China. Nanjing Center for Applied Mathematics, Nanjing, 211135, P.R.China. Corresponding author email: jjliu@seu.edu.cn.}

\date{\today}
\begin{document}
\begin{abstract}
    The Newtonian potential operator for the Helmholtz equation, which is represented by the volume integral with fundamental solution as kernel function,  is of great importance for direct and inverse scattering of acoustic waves. In this paper, the eigensystem for the Newtonian potential operator is firstly shown to be equivalent to that for the Helmholtz equation with nonlocal boundary condition  for a bounded and simply connected Lipschitz-regular domain. Then,  we compute explicitly the eigenvalues and  eigenfunctions of the Newtonian potential operator  when it is defined in a 3-dimensional ball. Furthermore, the eigenvalues' asymptotic behavior is demonstrated. To illustrate the behavior of certain eigenfunctions, some numerical simulations are included. 
\end{abstract}
\maketitle
 
\section{Introduction and statement of the results}
Denote by  
\begin{equation}\label{DefSolFund}
    \Phi_{k}(x,y) \, := \, \frac{e^{i \, k \, \left\vert x \, - \, y \right\vert}}{4 \, \pi  \, \left\vert x \, - \, y \right\vert}, \quad x \neq y 
\end{equation}
the fundamental solution to the Helmoltz equation in $\mathbb{R}^3$, which satisfies 
\begin{equation}\label{HelmholtzEquation}
    \Delta \Phi_{k}(x,y) \, + \, k^{2} \, \Phi_{k}(x,y) \, = \, - \, \delta(x,y)
\end{equation}
with wave number $0\not= k\in \mathbb{C}$,
where $\delta(\cdot,\cdot)$ is the Dirac delta function. 
For a bounded and simply connected Lipschitz-regular domain $\Omega\subset \mathbb{R}^3$, define the Newtonian potential operator $N^{k} : \mathbb{L}^{2}(\Omega) \longrightarrow \mathbb{H}^{2}(\Omega)$ by
\begin{eqnarray}\label{DefNk(f)}
 N^{k}(f)(\cdot) \, :=  \, \int_{\Omega} \Phi_{k}(\cdot,y) \, f(y) \, dy
\end{eqnarray}
for the density function $f(\cdot)\in \mathbb{L}^2(\Omega)$.  
 
In both scattering and potential theory, the Newtonian potential integral operator is of significant interest. Without going into too much detail, to learn about the properties and applications of the Newtonian potential, we recommend that readers view \cite{alsenafi2023foldy, ammari2019subwavelength, anderson1992spectral, EvansBook, gilbarg2001elliptic, cartan1945theorie, DGS-Bubbles, krasnoselskii1966integral} and \cite{mantile2024point}. To understand the connection between the integral Newtonian operator $N^{k}(\cdot)$ given by $(\ref{DefNk(f)})$ and its differential counterpart, it is obvious that the function 
\begin{equation*}
u \, := \, N^{k}(f)    
\end{equation*}
meets the following inhomogeneous Helmholtz equation
\begin{equation}\label{Eq1}
    \Delta u \, + \, k^{2} \, u \, = \, - \, f, \quad \text{in} \quad \Omega.  
\end{equation}
Therefore the study of the inhomogeneous Helmholtz equation $(\ref{Eq1})$ can largely be effected through the study of the Newtonian potential operator, applied to $f(\cdot)$, see (\ref{DefNk(f)}) for its definition, as well as the application of Newtonian potential operator in seeking the solution of $\nabla\cdot(\sigma\nabla u)=f$ under the framework of Levi function \cite{Liu1}. We refer to \cite{kalmenov2019, kellogg} and \cite{poincare} for additional information. The topic of Proposition \ref{MainProp} will be the reciprocal link, specifically how the integral operator (\ref{DefNk(f)}) can  and the differential operator (\ref{Eq1}). It is worth recalling that the Helmholtz Equation is a crucial partial differential equation that can be applied to various aspects of physics, such as optics, acoustics, electromagnetism's, and quantum mechanics, see \cite{filippi1983integral, wilcox2006scattering, colton2019inverse} and \cite{nedelec}. Due to the Helmholtz equation's role in various areas, significant research has been undertaken to examine its solutions' properties, such as asymptotic behavior (i.e., the behavior of $u(x)$ for $\left\vert x \right\vert \, \gg \, 1$), asymptotic behavior for small wave number (i.e., $k \ll 1$), spectrum analysis for the Helmholtz operator, efficient computing methods, etc. 
\medskip
\newline

Despite the presence of quantitative and qualitative results related to the Helmholtz differential operator $\left( \Delta \, + \, k^{2} \right)$ and/or the Newtonian potential operator $N^{k}(\cdot)$, to the best of our knowledge, there are still some questions that have not been addressed. It is unfortunate that there is little  literature on how to compute the dependability of the eigenvalues and eigenfunctions on the wave number $k$ and that it has mainly focused on the case $k=0$, i.e., 
\begin{equation*}
  - \, \Delta u \, = \, \frac{1}{\lambda} \, u  \quad \text{and/or} \quad N^{k=0}(u) \, = \, \lambda \, u,
\end{equation*}
see \cite{GN, anderson1992spectral} and \cite{Suragan}. 
\medskip
\newline 
Motivated by the explanations above, the aim of this paper is to compute the eigenvalues and the eigenfunctions related to the Newtonian operator $N^{k}(\cdot)$, given by $(\ref{DefNk(f)})$, and show how they are connected to the wave number $k$.
Based on the connection between the Newtonian operator and the Helmholtz equation with non local boundary condition, we propose an efficient scheme to compute the discrete eigenvalues of  $N^{k}(\cdot)$. To do this, inspired by the work of \cite{Suragan}, we establish the following proposition. 

\begin{proposition}\label{MainProp}
For any function $f \in \mathbb{L}^{2}(\Omega)$ and any wave number $k$ such that\footnote{By $\boldsymbol{\sigma}\left(\cdot\right)$ we denote the spectrum set.} $k^{2} \notin \boldsymbol{\sigma}\left(- \, \Delta^{D} \right)$, i.e. $k^{2}$ is not in the spectrum set of the Dirichlet problem for the Laplace operator,
$u(\cdot) = N^{k}(f)(\cdot)$ in $\Omega$ if and only if
 $u\in \mathbb{H}^{2}(\Omega)$ satisfies the inhomogeneous Helmholtz equation $(\ref{Eq1})$ together with the boundary condition
    \begin{equation}\label{BCSLDL}
    \frac{u(x)}{2} \, + \, \int_{\partial \Omega} \Phi_{k}(x,y) \, \partial_{\nu}u(y) \, d\sigma(y) \, - \, p.v. \int_{\partial \Omega} \frac{\partial \Phi_{k}}{\partial \nu(y)}(x,y) \, u(y) \, d\sigma(y) \, = \, 0, \quad x \in \partial \Omega. 
\end{equation}
\end{proposition}
\begin{proof}
We leave out some details and direct the reader to reference \cite[Proof of Theorem 2.1]{Suragan} for additional information. We start by assuming that $u$ is given by the Newtonian potential operator with density function $f \in \mathbb{L}^{2}(\Omega)$, i.e., 
    \begin{equation}\label{HAKT}
        u(x) \, = \, \int_{\Omega} \Phi_{k}(x,y) \, f(y) \, dy,  \quad x \in \Omega, 
    \end{equation} 
    see $(\ref{DefNk(f)})$. Then, by taking the Laplacian operator on its both sides, we deduce
    \begin{equation*}
        \Delta u(x) \, \overset{(\ref{HelmholtzEquation})}{=} \, - \, k^{2} \, \int_{\Omega} \Phi_{k}(x,y) \, f(y) \, dy \, - \, f(x) \, \overset{(\ref{HAKT})}{=} \, - \, k^{2} \, u(x) \, - \, f(x), 
    \end{equation*}
    or, equivalently, 
    \begin{equation*}
        \Delta u(x) \, + \, k^{2} \, u(x) \, = \, - \, f(x), \quad x \in \Omega. 
    \end{equation*}
    Hence, $(\ref{HAKT})$ becomes,       
    \begin{equation}
        u(x) \, = \, - \, \int_{\Omega} \Phi_{k}(x,y) \, \left( \Delta u(y) \, + \, k^{2} \, u(y) \right) \, dy,  \quad x \in \Omega, 
    \end{equation}
    which, by performing an integration by parts and using $(\ref{HelmholtzEquation})$, we deduce that
    \begin{equation}\label{SLDLu}
        SL^{k}_{\partial \Omega}\left( \partial_{\nu} u \right)(x) \, - \, DL^{k}_{\partial \Omega}(u)(x) \, = \, 0,  \quad x \in \Omega,  
    \end{equation} 
    where $SL^{k}_{\partial \Omega}(\cdot)$ is the single-layer operator defined by 
    \begin{equation*}
    SL^{k}_{\partial \Omega}(f)(x) \, := \, \int_{\partial \Omega} \Phi_{k}(x,y) \, f(y) \, d\sigma(y), \quad x \in \mathbb{R}^{3}, 
    \end{equation*}
    and 
    $DL^{k}_{\partial \Omega}(\cdot)$ is the double-layer operator defined by 
    \begin{equation*}
    DL^{k}_{\partial \Omega}(f)(x) \, := \, \int_{\partial \Omega}  \frac{\partial \Phi_{k}(x,y)}{\partial \, \nu(y)} \, f(y) \, d\sigma(y), \quad x \in \mathbb{R}^{3} \setminus \partial \Omega. 
    \end{equation*}
    For an in-depth study of the properties of the single-layer operator and the double-layer opertor, we refer interested readers to the references \cite{colton2013integral, chang2008, AmmariKang} and \cite{miyanishi2017}. Now, by letting $x \longrightarrow \partial \Omega$ for $x\in\Omega$ and using the continuity of the single-layer operator across the boundary and the jump formula for the double-layer operator, see \cite[Theorem 3.1]{colton2019inverse}, we obtain from (\ref{SLDLu}) the following boundary condition
    \begin{equation*}
     SL^{k}_{\partial \Omega}\left( \partial_{\nu} u \right)(x) \, + \, \frac{1}{2}u(x) \, - \, \mathcal{K}^{k}_{\partial \Omega}(u)(x) \, = \, 0,  \quad x \in \partial \Omega,
    \end{equation*}
    where $\mathcal{K}^{k}_{\partial \Omega}(\cdot)$ is the operator defined, from $\mathbb{L}^{2}(\partial \Omega)$ to $\mathbb{L}^{2}(\partial \Omega)$, by 
    \begin{equation*}
        \mathcal{K}^{k}_{\partial \Omega}(f)(x) \, := \, p.v. \int_{\partial \Omega} \frac{\partial \Phi_{k}(x,y)}{\partial \nu(y)} \, f(y) \, d\sigma(y), \quad a.e \; x \in \partial \Omega.
    \end{equation*}
Now we assume the existence of a function $u_{1} \in \mathbb{H}^{2}(\Omega)$ satisfying the Helmholtz equation $\Delta \, u_{1} \, + \, k^{2} \, u_{1} \, = \, - \, f$ in $\Omega$ with $f \in \mathbb{L}^{2}(\Omega)$, and the boundary condition 
    \begin{equation*}
        \frac{1}{2}u_{1} \, + \, SL^{k}_{\partial \Omega}\left( \partial_{\nu} u_{1} \right) \, - \, \mathcal{K}^{k}_{\partial \Omega}\left( u_{1} \right) \, = \, 0, \quad \text{on} \; \partial \Omega,
    \end{equation*}
    then we need to prove that $u_{1}(\cdot)$ is the Newtonian operator with density function $f(\cdot)$, i.e. 
    \begin{equation}\label{liu301}
        u_{1}(x)= \int_{\Omega} \Phi_{k}(x,y) f(y)  d\sigma(y). 
    \end{equation} 

Assume (\ref{liu301}) is not true. Then the function $v:= u-u_{1} \, \in \mathbb{H}^{2}(\Omega)$, where $u(x) =\int_{\Omega} \Phi_{k}(x,y) \, f(y) \, d\sigma(y)$ for $x \in \Omega$ satisfies 
    \begin{equation}\label{Equa-v-Omega}
    \Delta \, v \, + \, k^{2} \, v \, = \, 0, \quad \text{in} \; \;  \Omega,
    \end{equation}
    and the boundary condition 
    \begin{equation}\label{BCv}
        \frac{1}{2}v(x) \, - \, \mathcal{K}^{k}_{\partial \Omega}\left( v \right)(x) \, + \, SL^{k}_{\partial \Omega}\left( \partial_{\nu} v \right)(x) \, = \, 0, \quad  x \in \partial \Omega.  
    \end{equation}
    Besides, by multiplying $(\ref{Equa-v-Omega})$ with the fundamental solution $\Phi_{k}(\cdot, \cdot)$ and integrating by parts, we obtain 
    \begin{equation*}
        v(x) \, = \, - \, DL^{k}_{\partial \Omega}\left( v \right)(x) \, + \, SL^{k}_{\partial \Omega}\left( \partial_{\nu} v \right)(x), \quad x \in \Omega, 
    \end{equation*}
    which, by letting $x \longrightarrow \partial \Omega$ for $x\in\Omega$ and using the jump relation for the double-layer operator, we end up with the following relation 
    \begin{equation*}
        v(x) \, = \, \frac{v}{2}(x) \, - \, \mathcal{K}^{k}_{\partial \Omega}\left( v \right)(x) \, + \, SL^{k}_{\partial \Omega}\left( \partial_{\nu} v \right)(x), \quad x \in \partial \Omega. 
    \end{equation*}
    Thanks to $(\ref{BCv})$, the right hand side of the above equation is a vanishing term. Hence, $v(x) \, = \, 0$ on $\partial \Omega$. Finally, we obtain the following Dirichlet interior problem for the Helmholtz equation, 
    \begin{equation}\label{Eqv}
        \begin{cases}
        \Delta v \, + \, k^{2} \, v \, = \, 0, & \text{in $\Omega$}\\
            \qquad  \qquad v \, = \, 0, & \text{on $\partial \Omega$}
        \end{cases}.
    \end{equation}
    It is known in the literature, see for example \cite[Section 6.2.3]{EvansBook}, that under the condition $k^{2} \notin \boldsymbol{\sigma}\left(- \, \Delta\right)$, the problem $(\ref{Eqv})$, admits the unique solution $v = 0$, in $\Omega$. Hence, by construction of $v$, we deduce that 
    \begin{equation*} u(x) \, = \, u_{1}(x) \, =  \, \int_{\Omega} \Phi_{k}(x,y) \, f(y) \, dy, \quad x \in \Omega.
    \end{equation*}
    This ends the proof of Proposition \ref{MainProp}. 
\end{proof}
\begin{remark}
In the case of a simple domain, there is an explicit representation of the eigenvalues and eigenfunctions that are linked to the Laplacian operator, with the Dirichlet, Neumann, and Robin conditions. See \cite{GN} as well as the references therein. Therefore, for a simple domain, the condition $k^{2} \notin \boldsymbol{\sigma}\left(- \, \Delta\right)$, can be tested directly. 
\end{remark}
\medskip
Based on Proposition \ref{MainProp}, it is possible to derive the spectrum analysis of the Newtonian potential operator $N^{k}(\cdot)$ from the spectrum analysis of the Helmholtz operator, along with boundary condition $(\ref{BCSLDL})$. Now we explore the spectrum analysis of the Newtonian potential operator $N^{k}(\cdot)$ defined on a three dimensional ball, i.e., $\Omega \equiv B(0,\delta)$ with $\delta \in \mathbb{R}^{+}$. To do this, in $(\ref{HAKT})$, we let $f(\cdot) \, = \, - \, \lambda_{n}(k, \delta) \, u_{n}(\cdot)$ to obtain 
\begin{equation}\label{Equa-I}
    N^{k}(u_{n}) \, = \, - \, \frac{1}{\lambda_{n}(k, \delta)} \, u_{n}, \quad \text{in} \quad B(0,\delta),
\end{equation}
with $\lambda_{n}(k, \delta) \in \mathbb{C}$. After that, by using the result derived in Proposition \ref{MainProp}, can be equivalently expressed as
\begin{eqnarray}
\label{Equa-II}
   \Delta \, u_{n} \, + \, k^{2} \, u_{n} \, &=& \, \lambda_{n}(k, \delta) \, u_{n}, \quad \text{in} \; B(0,\delta),  \\ \label{Equa-III}
    \frac{1}{2}u_{n} \, + \, SL^{k}_{\partial B(0,\delta)}\left( \partial_{\nu} u_{n} \right) \, - \, \mathcal{K}^{k}_{\partial B(0,\delta)}\left(  u_{n} \right) \, &=& \, 0, \;\, \qquad \qquad \text{on} \;  \partial B(0,\delta). 
\end{eqnarray}
To prevent confusion in notation, as seen in $(\ref{Equa-I})$, we utilize the notation 
\begin{equation}\label{zetalambda}
    \zeta_{n}(k, \delta) \, := \, - \, \frac{1}{\lambda_{n}(k, \delta)},  
\end{equation}
to identify the eigenvalues connected to the Newtonian potential operator $N^{k}(\cdot)$.

Now the eigenvalues of $N^k(\cdot)$ can be represented by the following result. 

\begin{theorem}\label{MainThm}
    The eigenvalues of the Newtonian potential operator $N^{k}(\cdot)$ defined in $B(0,\delta)$, are given by 
    \begin{equation}\label{Eig-val-expression}
        \zeta_{n,j}\left( k, \delta \right) \, = \, \frac{\delta^{2}}{\left(\mu_{j}^{(n)}(k,\delta)\right)^{2} \, - \, \left(\delta \, k \right)^{2}},
    \end{equation}
    where $\mu_{j}^{(n)}(k,\delta)$ is the $j-$th positive root of the following transcendental equation
    \begin{equation}\label{Tr-Eq}
        J_{n+\frac{1}{2}}(x) \, - \, \frac{x}{\delta \, T(n,k,\delta) \, + \, (n + \frac{1}{2})} \, J_{n+\frac{3}{2}}(x) \, = \, 0,  
    \end{equation}
    with $J_{\nu}(\cdot)$ the Bessel function of the first kind of order $\nu \in \mathbb{R}$, and  $T(n,k,\delta)$ is given by
    \begin{equation}\label{T-Ref-Rem}
    T(n,k,\delta) \, := \, \frac{-1}{2 \, \delta} \, + \, 1 \, + \, \frac{i}{8} \, (-1)^{n} \, J_{n+\frac{1}{2}}(k \, \delta) \left[ H^{(1)}_{n+\frac{1}{2}}(k \, \delta) \, - \, 2 \, \delta\, k \, \left(H^{(1)}_{n+\frac{1}{2}}\right)^{\prime}(k \, \delta ) \right],
\end{equation}
where $H^{(1)}_{\nu}(\cdot)$ is the Hankel function of the first kind of order $\nu \in \mathbb{R}$. 

Moreover, the $2n+1$ linear independent eigenfunctions $v_{n,m,j}(\cdot)$ for $m=-n,\cdots,-1,0,1,\cdots, n$ corresponding to the eigenvalue $\zeta_{n,j}\left( k, \delta \right)$ are given by 
    \begin{equation}
        v_{n,j,m}(r, \theta, \phi, k, \delta) \, = \, \left[ \int_{0}^{\delta} \, r   \, \left\vert J_{n+\frac{1}{2}}\left(\mu^{(n)}_{j}(k,\delta)  \, \frac{r}{\delta} \right) \right\vert^{2} \, dr \right]^{-\frac{1}{2}} \, \frac{1}{\sqrt{r}} \, J_{n+\frac{1}{2}}\left(\mu_{j}^{(n)}(k,\delta) \,  \frac{r}{\delta}  \right) \, Y_{n}^{m}(\theta, \phi).
    \end{equation}
Besides, the set 
    $\left\{ v_{n,j,m}(\cdot):\; \left\vert m \right\vert \leq n,\quad n \in \mathbb{N} \right\}$
forms an orthonormal basis for the $L^{2}(B(0,\delta))$ space.
\end{theorem}
\begin{proof}
    See Section \ref{SPThm}.
\end{proof}
\medskip
The previous theorem suggests the following remark.
\begin{remark} 
Two remarks are in order.
\begin{enumerate}
    \item[]
    \item The term $T(n,k,\delta) \, \in \, \mathbb{C}$, see $(\ref{T-Ref-Rem})$, dictates that the solutions to $(\ref{Tr-Eq})$ should be in the complex plane, i.e., $\mu_{j}^{(n)}(k,\delta) \, \in \, \mathbb{C}$. Also, we are referring to a solution that has a positive real component when we say 'positive root', i.e., $Re\left(\mu_{j}^{(n)}(k,\delta)\right) \, \in \, \mathbb{R}^{+}$.
    \item[]
    \item \label{ZM(2)}
    The countability and disposition of the complex roots associated with $(\ref{Tr-Eq})$ in the complex plane are beyond the scope of this work. We recommend that readers refer to \cite{KS} and the references therein for an in-depth study on this question. 
\end{enumerate}
\end{remark}
The asymptotic behavior of the eigenvalues $\zeta_{n,j}(k, \delta)$, given by $(\ref{Eig-val-expression})$, for $n \gg 1$, can help to partially clarify Item $(\ref{ZM(2)})$ of the aforementioned remark, which concerns the countability and distribution of $\zeta_{n,j}(k, \delta)$. This is the objective for the next corollary. 
\begin{corollary}\label{MainCoro}
For $n$ fixed and large enough, i.e., $n \gg 1$, the family of eigenvalues $\left\{ \zeta_{n, j}\left( k, \delta \right) \right\}_{j \in \mathbb{N}}$ collapse to a single eigenvalue that we denote by $\zeta_{n}\left( k, \delta \right)$. Besides, the following asymptotic behavior holds,
   \begin{equation}\label{EquaCoro}
        \zeta_{n}\left( k, \delta \right) \, \sim \, \frac{\delta^{2}}{\dfrac{2n+3}{e} \, \left[ n \, + \, \delta \, \left[ 1 \, + \, \dfrac{(-1)^{n}}{4 \, \pi} \, + \, i \, \dfrac{(-1)^{n+1}}{8 \, \pi} \, \left( \dfrac{e \, k \, \delta}{2n+1} \right)^{2n+1} \right] \right] \, - \, \left(\delta \, k \right)^{2}}, 
\end{equation}
where $e$ is the Euler's number given by $e \, = \, 2.71828..$
\end{corollary}
\begin{proof}
    See Subsection \ref{SubSecProofCorollary}.
\end{proof}
\begin{figure}[H]
    \centering
    \includegraphics[scale=1]{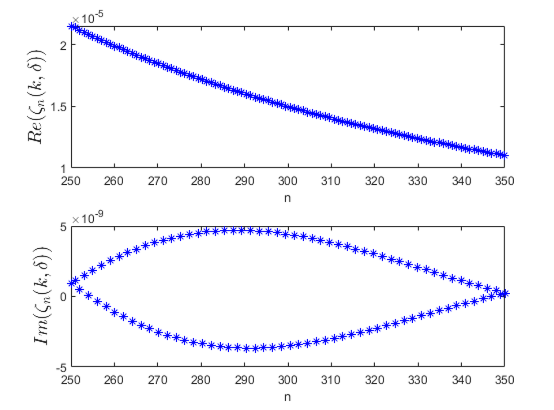}
    \caption{A schematic representation of the asymptotic behavior of the eigenvalues' $\zeta_{n}(k, \delta)$, with $250 \leq n \leq 350$, $k \, = \, 1+i$ and $\delta = 1$.}
    \label{Asymptotic-Behavior-Real+Imaginary-Part}
\end{figure}
The following remark is suggested by Figure \ref{Asymptotic-Behavior-Real+Imaginary-Part} and Corollary \ref{MainCoro}.
\begin{remark}
    For $n$ large enough, i.e., $n \gg 1$, the behavior of the eigenvalue $\zeta_{n}\left( k, \delta \right)$ is not affected by the wave number $k$ according to $(\ref{EquaCoro})$, as we have 
\begin{equation*}
        \zeta_{n}\left( k, \delta \right) \, \sim \, \frac{\delta^{2} \, e}{2 \, n^{2}}. 
\end{equation*}
\end{remark}
\section{Proof of Theorem \ref{MainThm}}\label{SPThm}
Based on Proposition \ref{MainProp}, we know the equivalence between the two problems $(\ref{Equa-I})$ and $(\ref{Equa-II})-(\ref{Equa-III})$. Next, we will focus on solving $(\ref{Equa-II})-(\ref{Equa-III})$.  
By the use of a spherical coordinates we represent $u_{n}(\cdot)$, solution of $(\ref{Equa-II})-(\ref{Equa-III})$,  as 
\begin{equation*}
    u_{n}(x) \, = \, u_{n}(r,\theta, \phi), \quad \text{with} \;\; r \in (0, \delta), \; \theta \in (0, \pi), \;\; \text{and} \;\; \phi \in (0, 2 \, \pi).   
\end{equation*}
Besides, we let
\begin{equation}\label{ccsc}
    x \, = \, \begin{pmatrix} 
    r \, \sin(\theta) \, \cos(\phi) \\
    r \, \sin(\theta) \, \sin(\phi) \\
    r \, \cos(\theta)
    \end{pmatrix} \quad \text{and} \quad y \, = \, \begin{pmatrix}
        \rho \, \sin(\alpha) \, \cos(\beta) \\
        \rho \, \sin(\alpha) \, \sin(\beta) \\
        \rho \, \cos(\alpha) 
    \end{pmatrix}. 
\end{equation}
The equation $(\ref{Equa-II})$ in  spherical coordinates becomes, 
\begin{eqnarray}\label{Eq4}
\nonumber
    \frac{1}{r^{2}} \, \frac{\partial}{\partial r} \left( r^{2} \, \frac{\partial u_{n}}{\partial r}(r,\theta, \phi) \right) \, &+& \,     \frac{1}{r^{2} \, \sin(\theta)} \, \frac{\partial}{\partial \theta} \left( \sin(\theta) \, \frac{\partial u_{n}}{\partial \theta}(r,\theta, \phi) \right) \\ &+& \,     \frac{1}{r^{2} \, \sin^{2}(\theta)} \, \frac{\partial^{2}\left( u_{n}\right)}{\partial^{2}\phi} (r,\theta, \phi)  \, = \, (\lambda_{n}(k, \delta) \, - \, k^{2} ) \, u_{n}(r,\theta, \phi),
\end{eqnarray}
and, the boundary condition given by $(\ref{Equa-III})$ becomes 
\begin{eqnarray}\label{Eq5}
\nonumber
    u_{n}(r, \theta, \phi)|_{r=\delta} \, &+& \, \frac{1}{2 \pi} \, \int_{0}^{\pi} \, \int_{0}^{2\pi} \frac{e^{i \, k \, \sqrt{\rho^{2}-2r\rho \Psi \, + \, r^{2}}}}{\sqrt{\rho^{2}-2r\rho \Psi \, + \, r^{2}}} \, \rho^{2} \, \sin(\alpha) \, \frac{\partial (u_{n})}{\partial \rho}(\rho, \alpha, \beta) |_{r = \rho =\delta}  \, d\alpha \, d\beta  \\
    &-& \, \frac{1}{2 \pi} \, \int_{0}^{\pi} \, \int_{0}^{2\pi} \frac{\partial}{\partial \rho}\left(\frac{e^{i \, k \, \sqrt{\rho^{2}-2r\rho \Psi \, + \, r^{2}}}}{\sqrt{\rho^{2}-2r\rho \Psi \, + \, r^{2}}} \right)\, \rho^{2} \, \sin(\alpha) \, u_{n}(\rho, \alpha, \beta) |_{r = \rho =\delta} \, d\alpha \, d\beta \, = \, 0,
\end{eqnarray}
with $\Psi$ is given by\footnote{The notation $\langle ; \rangle$ stands for the Euclidean inner product.} 
\begin{equation}\label{DefPsi}
    \Psi \, := \, \langle \hat{x}; \hat{y} \rangle \, \overset{(\ref{ccsc})}{=} \, \langle \begin{pmatrix} 
    \sin(\theta) \, \cos(\phi) \\
     \sin(\theta) \, \sin(\phi) \\
    \cos(\theta)
    \end{pmatrix} ; \begin{pmatrix}
        \sin(\alpha) \, \cos(\beta) \\
        \sin(\alpha) \, \sin(\beta) \\
        \cos(\alpha) 
    \end{pmatrix} \rangle .
\end{equation}
We seek a solution to $(\ref{Eq4})-(\ref{Eq5})$ of the form 
\begin{equation}\label{ASXD}
    u_{n}(r, \theta, \phi) \, = \, R_{n}(r) \, Y^{m}_{n}(\theta, \phi), \quad \text{with} \; \left\vert m \right\vert \leq n,  
\end{equation}
where $Y^{m}_{n}(\cdot, \cdot)$ is the spherical harmonic function of degree $n$ and order $m$, defined by 
\begin{equation}\label{Eq0705}
   Y^{m}_{n}(\theta, \phi) \, := \, \sqrt{\frac{(2n+1)}{4 \, \pi}} \, \sqrt{\frac{(n-m)!}{(n+m)!}} \, P_{n}^{m}(\cos(\theta)) \, e^{i \, m \, \phi}, 
\end{equation}
with $P_{n}^{m}(\cdot)$ is the associated Legendre polynomials of degree $n$ and order $m$. Besides, to mark the dependency of the solution $(\ref{ASXD})$ with respect to the order $m$, we set 
\begin{equation}\label{ASXDm}
    u_{n,m}(r, \theta, \phi) \, := \, R_{n}(r) \, Y^{m}_{n}(\theta, \phi), \quad \text{with} \; \left\vert m \right\vert \leq n. 
\end{equation}
Then, by plugging the above expression into $(\ref{Eq4})$, we derive the following equation
\begin{equation}\label{EDOR}
    \left( r^{2} \, R^{\prime}_{n}(r)  \right)^{\prime} \, + \, \left( \left(k^{2} \, -  \, \lambda_{n}(k, \delta) \right) \, r^{2} \, + \, \sigma_{n,m} \right) \, R_{n}(r) \, = \, 0,
\end{equation}
where $\sigma_{n,m}$ is given by 
\begin{equation}\label{Defsigma}
    \sigma_{n,m} \, :=  \, \frac{1}{Y^{m}_{n}(\theta, \phi) \, \sin(\theta)} \, \left[ \frac{\partial}{\partial \theta}\left(\sin(\theta) \, \frac{\partial \, (Y_{n}^{m})}{\partial \theta}(\theta, \phi) \right) \, + \, \frac{1}{\sin(\theta)} \, \frac{\partial^{2} \, (Y_{n}^{m})}{\partial^{2} \phi} (\theta, \phi) \right], 
\end{equation}
which, by using  the fact that\footnote{The formula $(\ref{Wiki})$ is well known in the literature, see \url{https://en.wikipedia.org/wiki/Associated_Legendre_polynomials}} 
\begin{equation}\label{Wiki}
    \frac{1}{\sin(\theta)} \, \frac{d}{d\theta}\left(\sin(\theta) \, \frac{d}{d\theta} \, P_{n}^{m}\left(\cos(\theta) \right)  \right) \, + \, \left( n \, \left( n + 1 \right) \, - \, \frac{m^{2}}{\sin^{2}(\theta)} \right) \, P_{n}^{m}\left(\cos(\theta) \right) \, = \, 0, 
\end{equation}
and the formula $(\ref{Eq0705})$ can be reduced to  
\begin{equation}\label{sigma=nn+1}
    \sigma_{n,m} \, = \, - \, n \, (n+1).  
\end{equation}
Hence, by using $(\ref{sigma=nn+1})$ into $(\ref{EDOR})$, we obtain the following Bessel's differential equation 
\begin{equation}\label{EDORnn+1}
    \left( r^{2} \, R^{\prime}_{n}(r)  \right)^{\prime} \, + \, \left( \left(k^{2} \, -  \, \lambda_{n}(k,\delta) \right) \, r^{2} \, - \, n \, (n+1) \right) \, R_{n}(r) \, = \, 0, 
\end{equation}
see \cite{bowman1958, abramowitz'sbook, watson1922} and \cite{erdelyi1953}, and the references therein for an in-depth study of Bessel's differential equation. The solution to $(\ref{EDORnn+1})$, will be given by\footnote{In $(\ref{Rn-Bessel-Jn})$, the notation $\sqrt{\cdot}$ stands for the square root of a complex number.}  
\begin{equation}\label{Rn-Bessel-Jn}
    R_{n}(r) \, = \, \sqrt{\frac{\pi}{2 \,  r }} \, J_{n+\frac{1}{2}}\left( \sqrt{k^{2}-\lambda_{n}(k, \delta)} \, r \right),
\end{equation}
where $J_{n \, + \,\frac{1}{2}}(\cdot)$ is the Bessel function of first kind of half-integer order.
\medskip
\newline 
The proof of Theorem \ref{MainThm} will be delayed until we announce and prove the technical lemma that follows. 
\begin{lemma}\label{Exp-Fund-Sol}
    The fundamental solution $\Phi_{k}(\cdot, \cdot)$ admits the following expansion, 
    \begin{equation}\label{EFS}
        \Phi_{k}(x,y) \, = \, \frac{i}{4} \, \sum_{n=0}^{\infty} (n+\frac{1}{2}) \, (-1)^{n} \, \frac{J_{n+\frac{1}{2}}(k|x|)}{\sqrt{|x|}} \, \frac{H^{(1)}_{n+\frac{1}{2}}(k|y|)}{\sqrt{|y|}} \, P_{n}(\Psi),  
    \end{equation}
    where $H^{(1)}_{\nu}(\cdot)$ is the Hankel function of the first kind of order $\nu \in \mathbb{R}$, and $\Psi$ is the parameter given by $(\ref{DefPsi})$. 
\end{lemma}
\begin{proof}
We start by recalling, from $(\ref{DefSolFund})$, that $\Phi_{k}(\cdot, \cdot)$ is given by 
\begin{eqnarray}\label{Equa0433}
\nonumber
    \Phi_{k}(x,y) \, &=& \, \frac{e^{i \, k \, \left\vert x \, - \, y \right\vert}}{4 \, \pi \, \left\vert x \, - \, y \right\vert} \, \\ &=&  \, \frac{i}{4} \, \sqrt{\frac{k}{2 \, \pi \,  \left\vert x \, -  \, y\right\vert}}  \, \left( J_{\frac{1}{2}}\left(k \, \left\vert x \, - \, y \right\vert \right) \, + \, i \, Y_{\frac{1}{2}}\left(k \, \left\vert x \, - \, y \right\vert \right) \right) \, = \, \frac{i \, k}{4 \, \sqrt{2 \, \pi}} \, \sqrt{\frac{1}{k \,  \left\vert x \, -  \, y\right\vert}}  \, H^{(1)}_{\frac{1}{2}}\left(k \, \left\vert x \, - \, y \right\vert \right), 
\end{eqnarray}
see \cite[Chapter 10]{abramowitz'sbook}. Besides, thanks to \cite[Formula (5.4)]{sommerfeld1943} the following expansion holds 
\begin{equation}\label{Equa0432}
    \frac{1}{\sqrt{k \, \left\vert x - y\right\vert}} \, H_{\frac{1}{2}}^{(1)}\left( k \, \left\vert x - y\right\vert \right) \, = \, \sqrt{2 \, \pi} \, \sum_{n=0}^{\infty} (n+\frac{1}{2}) \, (-1)^{n} \, \frac{J_{n+\frac{1}{2}}(k \, \left\vert x \right\vert)}{\sqrt{k \, \left\vert x \right\vert}} \, \frac{H^{(1)}_{n+\frac{1}{2}}(k \, \left\vert y \right\vert)}{\sqrt{k \, \left\vert y \right\vert}} P_{n}(\Psi).
\end{equation}
Then, by plugging $(\ref{Equa0432})$ into $(\ref{Equa0433})$, we derive $(\ref{EFS})$.  
This ends the proof of Lemma \ref{Exp-Fund-Sol}.  
\end{proof}
\medskip
Now, by using the representation $(\ref{ASXD})$, the expansion $(\ref{EFS})$ into the boundary condition $(\ref{Eq5})$, we obtain 
\begin{small}
\begin{eqnarray}
    \nonumber
    0 \, & = & \, R_{n}(\delta) \, Y_{n}^{m}\left(\theta, \phi \right) \\ \nonumber
    &+& \frac{1}{2 \, \pi} \, \int_{0}^{\pi} \, \int_{0}^{2 \pi} \frac{i}{4} \sum_{\ell=0}^{\infty} \, (\ell + \frac{1}{2}) \, (-1)^{\ell} \, \frac{J_{\ell+\frac{1}{2}}(k \, \delta)}{\sqrt{\delta}} \, \frac{H^{(1)}_{\ell+\frac{1}{2}}(k \, \delta)}{\sqrt{\delta}} \, P_{\ell}\left(\Psi \right) \, \delta^{2} \, \sin(\alpha) \, Y^{m}_{\ell}(\alpha, \beta) \, R_{n}^{\prime}(\delta) \, d\alpha \, d\beta  \\
    &-& \frac{1}{2 \pi} \, \int_{0}^{\pi} \, \int_{0}^{2\pi} \frac{i}{4} \sum_{\ell=0}^{\infty} \, \left( \ell + \frac{1}{2} \right) \, (-1)^{\ell} \, \frac{J_{\ell+\frac{1}{2}}(k \, \delta)}{\sqrt{\delta}} \, \frac{\partial }{\partial \rho} \left( \frac{H^{(1)}_{\ell+\frac{1}{2}}(k \, \rho)}{\sqrt{\rho}} \right)_{\big|_{\rho = \delta}} \, P_{\ell}(\Psi) \, \delta^{2} \, \sin(\alpha) \, R_{n}(\delta) \, Y_{\ell}^{m}(\alpha, \beta)\, d\alpha \, d\beta.  
\end{eqnarray}
\end{small}

Moreover, by using the fact that\footnote{$\boldsymbol{\delta}(\cdot, \cdot)$ is the Kronecker delta function.} 
\begin{equation*}
    \int_{0}^{\pi} \, \int_{0}^{2 \pi} Y^{m}_{\ell}(\alpha, \beta) \, P_{k}(\Psi) \, \sin(\alpha) \, d \alpha \, d\beta \, = \, \frac{4 \, \pi}{2 \ell \, + \, 1} \; Y^{m}_{\ell}(\theta, \phi) \; \boldsymbol{\delta}(\ell, k),
\end{equation*}
we deduce that 
\begin{eqnarray*}\label{Equa1226}
    \nonumber
    0 \, & = & \, R_{n}(\delta) \, Y^{m}_{n}\left(\theta, \phi \right) \, + \, \frac{1}{2 \, \pi} \,  \frac{i}{4}  \, (n + \frac{1}{2}) \, (-1)^{n} \, \frac{J_{n+\frac{1}{2}}(k \, \delta)}{\sqrt{\delta}} \, \frac{H^{(1)}_{n+\frac{1}{2}}(k \, \delta)}{\sqrt{\delta}} \,  \delta^{2} \,  \, Y^{m}_{n}(\alpha, \beta) \, R_{n}^{\prime}(\delta) \,  \\
    &-& \frac{1}{2 \pi} \,  \frac{i}{4}  \, \left( n + \frac{1}{2} \right) \, (-1)^{n} \, \frac{J_{n+\frac{1}{2}}(k \, \delta)}{\sqrt{\delta}} \, \frac{\partial }{\partial \rho} \left( \frac{H^{(1)}_{n+\frac{1}{2}}(k \, \rho)}{\sqrt{\rho}} \right)_{|_{\rho = \delta}}  \, \delta^{2} \,  \, R_{n}(\delta) \, Y^{m}_{n}(\alpha, \beta),   
\end{eqnarray*}
which, by taking $Y^{m}_{n}(\alpha, \beta)$ in front, can be reduced to 
\begin{eqnarray}\label{Equa1226}
    \nonumber
    0 \, & = & \, R_{n}(\delta) \, + \, \frac{1}{2 \, \pi} \,  \frac{i}{4}  \, (n + \frac{1}{2}) \, (-1)^{n} \, \frac{J_{n+\frac{1}{2}}(k \, \delta)}{\sqrt{\delta}} \, \frac{H^{(1)}_{n+\frac{1}{2}}(k \, \delta)}{\sqrt{\delta}} \,  \delta^{2} \,  \, R_{n}^{\prime}(\delta) \,  \\
    &-& \frac{1}{2 \pi} \,  \frac{i}{4}  \, \left( n + \frac{1}{2} \right) \, (-1)^{n} \, \frac{J_{n+\frac{1}{2}}(k \, \delta)}{\sqrt{\delta}} \, \frac{\partial }{\partial \rho} \left( \frac{H^{(1)}_{n+\frac{1}{2}}(k \, \rho)}{\sqrt{\rho}} \right)_{|_{\rho = \delta}}  \, \delta^{2} \,  R_{n}(\delta).  
\end{eqnarray}
Besides, 
\begin{equation}\label{Equa1227}
    \frac{\partial }{\partial \rho} \left( \frac{H^{(1)}_{n+\frac{1}{2}}(k \, \rho)}{\sqrt{\rho}} \right)_{|_{\rho = \delta}} \, = \, \frac{-1}{2} \, \frac{1}{\delta \, \sqrt{\delta}} \, H^{(1)}_{n+\frac{1}{2}}(k \, \delta) \, - \, \frac{1}{\sqrt{\delta}} \, k \, \left(H^{(1)}_{n+\frac{1}{2}}\right)^{\prime}(k \, \delta). 
\end{equation}
Then, by plugging $(\ref{Equa1227})$ into $(\ref{Equa1226})$, we obtain 
\begin{eqnarray}\label{Equa1230}
    \nonumber
    0 \, & = & \, R_{n}(\delta) \, \left[ 1 \, + \, \frac{1}{2 \pi} \,  \frac{i}{4}  \, \left( n + \frac{1}{2} \right) \, (-1)^{n} \, \frac{J_{n+\frac{1}{2}}(k \, \delta)}{\sqrt{\delta}} \, \left[ \frac{1}{2} \, \frac{1}{\delta \, \sqrt{\delta}} \, H^{(1)}_{n+\frac{1}{2}}(k \, \delta) \, + \, \frac{1}{\sqrt{\delta}} \, k \, \left(H^{(1)}_{n+\frac{1}{2}}\right)^{\prime}(k \, \delta )\right]   \, \delta^{2} \right] \\ &+& \, \frac{1}{2 \, \pi} \,  \frac{i}{4}  \, (n + \frac{1}{2}) \, (-1)^{n} \, \frac{J_{n+\frac{1}{2}}(k \, \delta)}{\sqrt{\delta}} \, \frac{H^{(1)}_{n+\frac{1}{2}}(k \, \delta)}{\sqrt{\delta}} \,  \delta^{2} \,  \, R_{n}^{\prime}(\delta).   
\end{eqnarray}
The equation $(\ref{Rn-Bessel-Jn})$, should satisfies the condition $(\ref{Equa1230})$. This implies, 
\begin{eqnarray*}
    0 \, &=& \, \sqrt{\frac{\pi}{2 \delta}} \, J_{n+\frac{1}{2}}\left(\sqrt{k^{2}-\lambda_{n}(k, \delta)} \, \delta \right) \, \left[1 \, + \, \frac{i}{8} \, (-1)^{n} \, J_{n+\frac{1}{2}}(k \, \delta) \left[ H^{(1)}_{n+\frac{1}{2}}(k \, \delta) \, - \, 2 \, \delta\, k \, \left(H^{(1)}_{n+\frac{1}{2}}\right)^{\prime}(k \, \delta ) \right]\right] \\
    &+& \, \sqrt{\frac{\pi}{2 \delta}} \, \left[ \frac{-1}{2 \, \delta} \, J_{n+\frac{1}{2}}\left(\sqrt{k^{2}-\lambda_{n}(k, \delta)} \, \delta \right) \, + \, \sqrt{k^{2}-\lambda_{n}(k, \delta)} \, J_{n+\frac{1}{2}}^{\prime}\left(\sqrt{k^{2}-\lambda_{n}(k, \delta)} \, \delta \right) \right],
\end{eqnarray*}
which can be rewritten as 
\begin{equation}\label{Eig-val-Eq}
    J_{n+\frac{1}{2}}\left(\sqrt{k^{2}-\lambda_{n}(k, \delta)} \, \delta\right) \, + \, \frac{\sqrt{k^{2}-\lambda_{n}(k, \delta)}}{T(n,k,\delta)} \, J^{\prime}_{n+\frac{1}{2}}\left(\sqrt{k^{2}-\lambda_{n}(k, \delta)} \, \delta \right) \, = \, 0,
\end{equation}
where $T(n,k,\delta)$ is the term given by 
\begin{equation*}
    T(n,k,\delta) \, := \, \frac{-1}{2 \, \delta} \, + \, 1 \, + \, \frac{i}{8} \, (-1)^{n} \, J_{n+\frac{1}{2}}(k \, \delta) \left[ H^{(1)}_{n+\frac{1}{2}}(k \, \delta) \, - \, 2 \, \delta\, k \, \left(H^{(1)}_{n+\frac{1}{2}}\right)^{\prime}(k \, \delta ) \right].
\end{equation*}
For $k$ known and $\delta$ known, we can compute explicitly the term $T(n,k,\delta)$. Hence, the term $T(n,k,\delta)$ is known. The solutions of equation $(\ref{Eig-val-Eq})$ have been extensively investigated in the references \cite{Landau}. Applying the differentiability property associated with the Bessel function, provided by the following formula
\begin{equation}\label{DiffBesselFct}
    J_{\nu}^{\prime}(x) \, = \, \frac{\nu}{x} \, J_{\nu}(x) \, - \, J_{\nu + 1}(x), \quad \nu \in \mathbb{C}, 
\end{equation}
we rewrite $(\ref{Eig-val-Eq})$ as,
\begin{equation}\label{TranscendentalEquation}
    J_{n+\frac{1}{2}}\left(\sqrt{k^{2}-\lambda_{n}(k, \delta)} \, \delta \right) \, - \, \frac{\sqrt{k^{2}-\lambda_{n}(k, \delta)} \, \delta}{\delta \, T(n,k,\delta) \, + \, (n + \frac{1}{2})} \, J_{n+\frac{3}{2}}\left(\sqrt{k^{2}-\lambda_{n}(k, \delta)} \, \delta \right) \, = \, 0. 
\end{equation}
Consequently, from $(\ref{TranscendentalEquation})$, we see that the eigenvalue, depending on the index $n$, i.e., $\lambda \, = \, \lambda_{n}(k,\delta)$  will be given by
\begin{equation}\label{lnkd}
    \lambda_{n}(k,\delta) \, = \, k^{2} \, - \, \left( \frac{\mu_{j}^{(n)}(k,\delta)}{\delta} \right)^{2},
\end{equation}
where $\mu_{j}^{(n)}(k,\delta)$ is the $j^{th}$ positive root of the following transcendental equation, 
\begin{equation*}
    J_{n+\frac{1}{2}}\left(\mu_{j}^{(n)}(k,\delta)\right) \, - \, \frac{\mu_{j}^{(n)}(k,\delta)}{\delta \, T(n,k,\delta) \, + \, (n + \frac{1}{2})} \, J_{n+\frac{3}{2}}\left(\mu_{j}^{(n)}(k,\delta)\right) \, = \, 0. 
\end{equation*}
Moreover, to mark the dependency of the eigenvalue $\lambda_{n}(k, \delta)$ with respect to the index $j$, we note 
\begin{equation}\label{lnjkd}
    \lambda_{n,j}(k,\delta) \, := \, k^{2} \, - \, \left( \frac{\mu_{j}^{(n)}(k,\delta)}{\delta} \right)^{2},
\end{equation}
and, by recalling $(\ref{zetalambda})$, we obtain 
\begin{equation*}
        \zeta_{n,j}\left( k, \delta \right) \, = \, \frac{\delta^{2}}{\left(\mu_{j}^{(n)}(k,\delta)\right)^{2} \, - \, \left(\delta \, k \right)^{2}}.
\end{equation*}
Finally, by returning to $(\ref{ASXDm})$, using $(\ref{Rn-Bessel-Jn})$ and $(\ref{lnjkd})$, we deduce the following expression for the eigenfunction associated to the eigenvalue $\lambda_{n,j}(k,\delta)$. More precisely, we have 
\begin{equation}\label{EigFct1stCase}
    u_{n,m,j}(r, \theta, \phi, k, \delta) \, = \, \sqrt{\frac{\pi}{2 \, r}} \, J_{n+\frac{1}{2}}\left(\mu^{(n)}_{j}(k,\delta)  \, \frac{r}{\delta} \right) \, Y_{n}^{m}(\theta, \phi).
\end{equation}
\medskip
Now, by computing the $\mathbb{L}^{2}(B(0,\delta))$-inner product between two arbitrary eigenfunctions, we obtain 
\begin{eqnarray*}
\langle u_{n,j,m}; u_{n^{\prime},j^{\prime},m^{\prime}} \rangle_{\mathbb{L}^{2}(B(0,\delta))} \, &=& \, \int_{B(0,\delta)} u_{n,j,m}(x) \, \overline{u_{n^{\prime},j^{\prime},m^{\prime}}(x)} \, dx \\
&=& \, \int_{0}^{\delta} \, \int_{0}^{\pi} \, \int_{0}^{2 \, \pi} u_{n,j,m}(r, \theta, \phi) \, \overline{u_{n^{\prime},j^{\prime},m^{\prime}}(r, \theta, \phi)} \, r^{2} \, \sin(\theta) \, d\phi  \, d\theta \, dr \\
& \overset{(\ref{EigFct1stCase})}{=} & \frac{\pi}{2} \, \int_{0}^{\delta} r \, J_{n+\frac{1}{2}}\left(\mu^{(n)}_{j}(k,\delta)  \, \frac{r}{\delta} \right) \, \overline{J_{n^{\prime}+\frac{1}{2}}\left(\mu^{(n^{\prime})}_{j^{\prime}}(k,\delta)  \, \frac{r}{\delta} \right)} \, dr \, \langle Y_{n}^{m}; Y_{n^{\prime}}^{m^{\prime}} \rangle_{\mathbb{L}^{2}(\mathbb{S}^{2})}, 
\end{eqnarray*}
where $\mathbb{S}^{2}$ is the unit sphere in $\mathbb{R}^{3}$. In addition, by knowing the following orthogonal relation satisfied by the spherical harmonics,
\begin{equation*}
    \langle Y_{n}^{m}; Y_{n^{\prime}}^{m^{\prime}} \rangle_{\mathbb{L}^{2}(\mathbb{S}^{2})} \, = \, \boldsymbol{\delta}(n, n^{\prime}) \, \boldsymbol{\delta}(m, m^{\prime})
\end{equation*}
we deduce 
\begin{equation}\label{orth-eig-fct}
\langle u_{n,j,m}; u_{n^{\prime},j^{\prime},m^{\prime}} \rangle_{\mathbb{L}^{2}(B(0,\delta))} \,  = \, \frac{\pi}{2} \, \int_{0}^{\delta} r \, J_{n+\frac{1}{2}}\left(\mu^{(n)}_{j}(k,\delta)  \, \frac{r}{\delta} \right) \, \overline{J_{n^{\prime}+\frac{1}{2}}\left(\mu^{(n^{\prime})}_{j^{\prime}}(k,\delta)  \, \frac{r}{\delta} \right)} \, dr \, \boldsymbol{\delta}(n, n^{\prime}) \, \boldsymbol{\delta}(m, m^{\prime}).  
\end{equation}
Hence, from $(\ref{orth-eig-fct})$, it is clear that the eigenfunctions $u_{n,j,m}(\cdot)$, with $n \in \mathbb{N} \; \text{and} \; \left\vert m \right\vert \leq n$, are mutually orthogonal. Besides, the completeness of the set 
\begin{equation*}
    \left\{ u_{n,j,m}(\cdot), \; \text{with} \; n \in \mathbb{N} \; \text{and} \; \left\vert m \right\vert \leq n \right\},
\end{equation*}
is a consequence of the completeness of the spherical harmonics set 
\begin{equation*}
 \left\{ Y_{n}^{m}, \; \text{with} \; n \in \mathbb{N} \; \text{and} \; \left\vert m \right\vert \leq n \right\}.   
\end{equation*}
After that, to get an orthonormal sequence, we need to compute 
\begin{eqnarray}\label{Equa0603}
\nonumber
    \left\Vert u_{n,m,j} \right\Vert_{\mathbb{L}^{2}(B(0,\delta))} \; &:=& \; \left[ \int_{B(0,\delta)} \left\vert u_{n,m,j}(x) \right\vert^{2} \, dx \right]^{\frac{1}{2}} \\ \nonumber &=& \; \left[ \int_{0}^{\delta} \, \int_{0}^{\pi} \, \int_{0}^{2 \, \pi} \left\vert u_{n,m,j}(r, \theta, \phi, k, \delta) \right\vert^{2} \, r^{2} \, \sin(\theta) \, d\phi \, d\theta \, dr \right]^{\frac{1}{2}} \\ \nonumber
    & \overset{(\ref{EigFct1stCase})}{=} & \; \sqrt{\frac{\pi}{2}} \; \left[ \int_{0}^{\delta} \, r   \, \left\vert J_{n+\frac{1}{2}}\left(\mu^{(n)}_{j}(k,\delta)  \, \frac{r}{\delta} \right) \right\vert^{2} \, dr \right]^{\frac{1}{2}} \, \left[\int_{0}^{\pi} \, \int_{0}^{2 \, \pi} \, \left\vert Y_{n}^{m}(\theta, \phi) \right\vert^{2}  \sin(\theta) \, d\phi \, d\theta  \right]^{\frac{1}{2}} \\ \nonumber
    & \overset{(\ref{Eq0705})}{=} & \; \frac{\sqrt{\pi \, (2n+1)}}{2} \, \sqrt{\frac{(n-m)!}{(n+m)!}} \,  \left[ \int_{0}^{\delta} \, r   \, \left\vert J_{n+\frac{1}{2}}\left(\mu^{(n)}_{j}(k,\delta)  \, \frac{r}{\delta} \right) \right\vert^{2} \, dr \right]^{\frac{1}{2}} \, \left[\int_{-1}^{1}  \, \left\vert P_{n}^{m}(x) \right\vert^{2} \, d x  \right]^{\frac{1}{2}} \\
    & = & \; \sqrt{\frac{\pi}{2}} \,  \left[ \int_{0}^{\delta} \, r   \, \left\vert J_{n+\frac{1}{2}}\left(\mu^{(n)}_{j}(k,\delta)  \, \frac{r}{\delta} \right) \right\vert^{2} \, dr \right]^{\frac{1}{2}}.
\end{eqnarray}
Hence, by setting
\begin{equation*}
    v_{n,m,j}(r, \theta, \phi, k, \delta) \, := \frac{u_{n,m,j}(r, \theta, \phi, k, \delta)}{\left\Vert u_{n,m,j} \right\Vert_{\mathbb{L}^{2}(B(0,\delta))}},
\end{equation*}
and, using $(\ref{EigFct1stCase})$ and $(\ref{Equa0603})$, we obtain 
\begin{equation*}
     v_{n,m,j}(r, \theta, \phi, k, \delta) \, = \, \left[ \int_{0}^{\delta} \, r   \, \left\vert J_{n+\frac{1}{2}}\left(\mu^{(n)}_{j}(k,\delta)  \, \frac{r}{\delta} \right) \right\vert^{2} \, dr \right]^{-\frac{1}{2}} \, \frac{1}{\sqrt{r}} \, J_{n+\frac{1}{2}}\left(\mu^{(n)}_{j}(k,\delta)  \, \frac{r}{\delta} \right) \, Y_{n}^{m}(\theta, \phi).
\end{equation*}
Then, by construction, the set  
\begin{equation*}
    \left\{ v_{n,j,m}(\cdot), \; \text{with} \; n \in \mathbb{N} \; \text{and} \; \left\vert m \right\vert \leq n \right\},
\end{equation*}
form a basis on $\mathbb{L}^{2}(B(0,\delta))$. 
This ends the proof of Theorem \ref{MainThm}.
\subsection{Proof of Corollary \ref{MainCoro}}\label{SubSecProofCorollary} 
The goal of this subsection is to investigate the asymptotic behavior of the eigenvalues $\left\{ \zeta_{n,j}\left( k, \delta \right) \right\}_{n,j \in \mathbb{N}}$, with respect to the index $n$, for $n \gg 1$. To do this, from $(\ref{Eig-val-expression})$, we recall that 
\begin{equation}\label{Eig-Val-Expression-SubSec}
        \zeta_{n,j}\left( k, \delta \right) \, = \, \frac{\delta^{2}}{\left(\mu_{j}^{(n)}(k,\delta)\right)^{2} \, - \, \left(\delta \, k \right)^{2}}, 
\end{equation}
where $\mu_{j}^{(n)}(k,\delta)$ is the $j^{th}$ positive root of  
\begin{equation}\label{Tr-Eq-SubSec}
        J_{n+\frac{1}{2}}(x) \, - \, \frac{x}{\delta \, T(n,k,\delta) \, + \, (n + \frac{1}{2})} \, J_{n+\frac{3}{2}}(x) \, = \, 0,  
    \end{equation}
    with $T(n,k,\delta)$ is the term given by
    \begin{equation}\label{T-Ref-Rem-SubSec}
    T(n,k,\delta) \, := \, \frac{-1}{2 \, \delta} \, + \, 1 \, + \, \frac{i}{8} \, (-1)^{n} \, J_{n+\frac{1}{2}}(k \, \delta) \left[ H^{(1)}_{n+\frac{1}{2}}(k \, \delta) \, - \, 2 \, \delta\, k \, \left(H^{(1)}_{n+\frac{1}{2}}\right)^{\prime}(k \, \delta ) \right]. 
\end{equation}
Hence, to derive the asymptotic behavior of the eigenvalues $\left\{ \zeta_{n,j}\left( k, \delta \right) \right\}_{n,j \in \mathbb{N}}$, given by $(\ref{Eig-Val-Expression-SubSec})$, it is necessary to derive the asymptotic behavior of the term $T(n,k,\delta)$, given by $(\ref{T-Ref-Rem-SubSec})$, and the asymptotic behavior of the equation $(\ref{Tr-Eq-SubSec})$. To accomplish this, we split the study into two steps. 
\begin{enumerate}
    \item[]
    \item Asymptotic behavior of the term $T(n,k,\delta)$. \\ 
    Thanks to \cite[Section 9.3]{abramowitz'sbook}, for $\nu \gg 1$, we know that
    \begin{equation}\label{AsBehJY}
        J_{\nu}(x) \, \sim \, \frac{1}{\sqrt{2 \, \pi \, \nu}} \, \left( \frac{e \, x}{2 \, \nu} \right)^{\nu} \quad \text{and} \quad Y_{\nu}(x) \, \sim \, - \, \sqrt{\frac{2}{\pi \, \nu}} \, \left( \frac{e \, x}{2 \, \nu} \right)^{- \, \nu},
    \end{equation}
    where $Y_{\nu}(\cdot)$ is the Bessel function of the second kind of order $\nu$. In addition, we know that\footnote{This is a well know formula, see \url{https://en.wikipedia.org/wiki/Bessel_function}} 
    \begin{equation}\label{DefH=J+Y}
        H^{(1)}_{\nu}(\cdot) \, := \, J_{\nu}(\cdot) \, + \, i \, Y_{\nu}(\cdot).  
    \end{equation}
    Hence, for $\nu \gg 1$, using $(\ref{AsBehJY})$ and $(\ref{DefH=J+Y})$, we obtain 
    \begin{equation}\label{AsBehH}
        H^{(1)}_{\nu}(x) \, \sim \, \frac{1}{\sqrt{2 \, \pi \, \nu}} \, \left( \frac{e \, x}{2 \, \nu} \right)^{\nu}  \, - \, i \, \sqrt{\frac{2}{\pi \, \nu}} \, \left( \frac{e \, x}{2 \, \nu} \right)^{- \, \nu},
    \end{equation}
    which, by taking its derivative with respect to the variable $x$, implies
        \begin{equation}\label{AsBehHprime}
        \left( H^{(1)}_{\nu} \right)^{\prime}(x) \, \sim \, \frac{e}{2 \, \sqrt{2 \, \pi \, \nu}} \, \left( \frac{e \, x}{2 \, \nu} \right)^{(\nu - 1)}  \, +  \, \frac{i \, e}{\sqrt{2 \, \pi \, \nu}} \, \left( \frac{e \, x}{2 \, \nu} \right)^{-(\nu+1)}.
    \end{equation}
    Now, by returning back to $(\ref{T-Ref-Rem-SubSec})$, using $(\ref{AsBehJY}), (\ref{AsBehH})$ and $(\ref{AsBehHprime})$, with $\nu$ being $\nu \, = \, n  + \frac{1}{2}$, we can deduce the following asymptotic behavior of the term $T(n,k,\delta)$, 
    \begin{equation*}
        T(n,k,\delta) \, \sim \, - \, \frac{1}{2 \, \delta} \, + \, 1 \, + \, \frac{(-1)^{n} \, (n+1)}{2 \, \pi \, (2n+1)} \, + \, i \, \frac{(-1)^{n+1} \, n}{4 \, \pi \, (2n+1)} \, \left( \frac{e \, k \, \delta}{2n+1} \right)^{2n+1}, 
    \end{equation*}
    which, by knowing that $n \gg 1$, can be reduced to
        \begin{equation}\label{AsyBehTnkdelta}
        T(n,k,\delta) \, \sim \, - \, \frac{1}{2 \, \delta} \, + \, 1 \, + \, \frac{(-1)^{n}}{4 \, \pi} \, + \, i \, \frac{(-1)^{n+1}}{8 \, \pi} \, \left( \frac{e \, k \, \delta}{2n+1} \right)^{2n+1}.
        \end{equation}
    \item[]
    \item Asymptotic behavior of the equation $(\ref{Tr-Eq-SubSec})$. \\
    By going back to $(\ref{Tr-Eq-SubSec})$, using $(\ref{AsBehJY})$ and $(\ref{AsyBehTnkdelta})$, we can obtain the asymptotic behavior equation related to the transcendental equation
    \begin{equation}\label{ABTE}
        \left( \frac{e \, x}{2n+1} \right)^{n+\frac{1}{2}} \, - \, \frac{x}{n \, + \, \delta \, \left[ 1 \, + \, \dfrac{(-1)^{n}}{4 \, \pi} \, + \, i \, \dfrac{(-1)^{n+1}}{8 \, \pi} \, \left( \dfrac{e \, k \, \delta}{2n+1} \right)^{2n+1} \right]} \,  \left( \frac{e \, x}{2n+3} \right)^{n+\frac{3}{2}} \, = \, 0.
    \end{equation}
    \end{enumerate}
    Then, by knowing that $n \gg 1$, solving $(\ref{ABTE})$ and denoting the obtained solution(s) by $\mu_{j}^{(n)}(k,\delta)$, we deduce that 
    \begin{equation*}
        \left( \mu_{j}^{(n)}(k,\delta) \right)^{2} \, \sim \, \frac{2n+3}{e} \, \left[ n \, + \, \delta \, \left[ 1 \, + \, \dfrac{(-1)^{n}}{4 \, \pi} \, + \, i \, \dfrac{(-1)^{n+1}}{8 \, \pi} \, \left( \dfrac{e \, k \, \delta}{2n+1} \right)^{2n+1} \right] \right]. 
    \end{equation*}
    Observe that $(\ref{ABTE})$ admits only one positive  solution. Then, to note short, we set $\mu^{(n)}(k,\delta)$ instead of $\mu_{j}^{(n)}(k,\delta)$, and we get   
        \begin{equation}\label{ABRTE}
        \left( \mu^{(n)}(k,\delta) \right)^{2} \, \sim \, \frac{2n+3}{e} \, \left[ n \, + \, \delta \, \left[ 1 \, + \, \dfrac{(-1)^{n}}{4 \, \pi} \, + \, i \, \dfrac{(-1)^{n+1}}{8 \, \pi} \, \left( \dfrac{e \, k \, \delta}{2n+1} \right)^{2n+1} \right] \right]. 
    \end{equation}
   Consequently, by returning to $(\ref{Eig-Val-Expression-SubSec})$ and plugging $(\ref{ABRTE})$, we deduce the following asymptotic behavior,  
   \begin{equation}
        \zeta_{n}\left( k, \delta \right) \, \sim \, \frac{\delta^{2}}{\dfrac{2n+3}{e} \, \left[ n \, + \, \delta \, \left[ 1 \, + \, \dfrac{(-1)^{n}}{4 \, \pi} \, + \, i \, \dfrac{(-1)^{n+1}}{8 \, \pi} \, \left( \dfrac{e \, k \, \delta}{2n+1} \right)^{2n+1} \right] \right] \, - \, \left(\delta \, k \right)^{2}}. 
\end{equation}
This ends the Proof of Corollary \ref{MainCoro}.
\section{Numerical Computations}
This section is designed to present numerical illustrations that pertain to the Newtonian potential eigensystem (eigenvalues and eigenfunctions), i.e. $N^{k}(\cdot)$. We start by recalling, from Theorem \ref{MainThm}, that the eigenvalues of the Newtonian potential $N^{k}(\cdot)$, defined over the ball $B(0,\delta)$, are given by 
\begin{equation}\label{Num-Eq1}
        \zeta_{n,j}\left( k, \delta \right) \, = \, \frac{\delta^{2}}{\left(\mu_{j}^{(n)}(k,\delta)\right)^{2} \, - \, \left(\delta \, k \right)^{2}},
    \end{equation}
and their corresponding eigenfunctions $v_{n,m,j}(\cdot, \cdot, \cdot)$ are given by    
\begin{equation}\label{Num-Eq2}
        v_{n,j,m}(r, \theta, \phi, k, \delta) \, = \, \left[ \int_{0}^{\delta} \, r   \, \left\vert J_{n+\frac{1}{2}}\left(\mu^{(n)}_{j}(k,\delta)  \, \frac{r}{\delta} \right) \right\vert^{2} \, dr \right]^{-\frac{1}{2}} \, \frac{1}{\sqrt{r}} \, J_{n+\frac{1}{2}}\left(\mu_{j}^{(n)}(k,\delta) \,  \frac{r}{\delta}  \right) \, Y_{n}^{m}(\theta, \phi), \; \text{with} \; \left\vert m \right\vert \, \leq \, n.
    \end{equation}
Clearly, from $(\ref{Num-Eq1})$ and $(\ref{Num-Eq2})$, the computation of $\left(\zeta_{n,j}\left( k, \delta \right),  v_{n,j,m}(r, \theta, \phi, k, \delta)\right)$ is directly related to the computation of the root $\mu_{j}^{(n)}(k,\delta)$. We recall, from $(\ref{Tr-Eq})$, that $\mu_{j}^{(n)}(k,\delta)$ is the $j^{th}$ positive root of the following Transcendental Equation
   \begin{equation}\label{Num-Eq3}
        J_{n+\frac{1}{2}}(x) \, - \, \frac{x}{\delta \, T(n,k,\delta) \, + \, (n + \frac{1}{2})} \, J_{n+\frac{3}{2}}(x) \, = \, 0, 
    \end{equation}
    where $T(n,k,\delta)$ is the term given by
    \begin{equation}\label{Num-Eq4}
    T(n,k,\delta) \, := \, \frac{-1}{2 \, \delta} \, + \, 1 \, + \, \frac{i}{8} \, (-1)^{n} \, J_{n+\frac{1}{2}}(k \, \delta) \left[ H^{(1)}_{n+\frac{1}{2}}(k \, \delta) \, - \, 2 \, \delta\, k \, \left(H^{(1)}_{n+\frac{1}{2}}\right)^{\prime}(k \, \delta ) \right].
   \end{equation}
To summarize, our algorithm on the computation of $\left(\zeta_{n,j}\left( k, \delta \right),  v_{n,j,m}(r, \theta, \phi, k, \delta)\right)$ goes as follows. Firstly, for a given $\left(n, k, \delta \right) \in \mathbb{N} \times \mathbb{C} \times \mathbb{R}^{+}$, we compute the term $T\left(n, k, \delta \right)$, given by $(\ref{Num-Eq4})$. Next, by plugging the obtained value of $T\left(n, k, \delta \right)$ in $(\ref{Num-Eq3})$, we compute the positive roots $\mu_{j}^{(n)}(k,\delta)$, for $j=1,2,3,\cdots$. Lastly, we use the obtained values for $\mu_{j}^{(n)}(k,\delta)$ to compute the eigensystem $\left(\zeta_{n,j}\left( k, \delta \right),  v_{n,j,m}(r, \theta, \phi, k, \delta)\right)$. In addition, as indicated by $(\ref{Num-Eq1})-(\ref{Num-Eq2})$, the eigensystem $\left(\zeta_{n,j}\left( k, \delta \right),  v_{n,j,m}(r, \theta, \phi, k, \delta)\right)$ depends on both the index $n$, the wave number $k$ and $\delta$, the radius of the ball where the Newtonian potential $N^{k}(\cdot)$ is defined. In order to determine how the eigensystem $\left(\zeta_{n,j}\left( k, \delta \right),  v_{n,j,m}(r, \theta, \phi, k, \delta)\right)$ behaves with respect to each parameter, we split the study into three independent parts. In each part, we vary one parameter while keeping the two others fixed\footnote{For the calculations mentioned in the tables below, we only retain the initial four numbers after the decimal point.}. 
\begin{enumerate}
    \item Eigensystem behavior with respect to the index $n$. 
\begin{table}[H]
\begin{tabular}{|l|ll|ll|}
\hline
                       & \multicolumn{2}{c|}{Root $\mu_{j}^{(n)}$}                     & \multicolumn{2}{c|}{Eigenvalue $\zeta_{n,j}$}                        \\ \hline
$n=0$ & \multicolumn{1}{l|}{$\mu_{j=1}^{(n=0)}(k=2,\delta=1)$} & 1.6364 + 0.0739i  & \multicolumn{1}{l|}{$\zeta_{n=0,j=1}(k=2, \delta = 1)$} & - 0.7290 - 0.1328i \\ \cline{2-5} 
                       & \multicolumn{1}{l|}{$\mu_{j=2}^{(n=0)}(k=2,\delta=1)$} & 4.7340 + 0.0265i  & \multicolumn{1}{l|}{$\zeta_{n=0,j=2}(k=2, \delta = 1)$} & 0.0543 - 0.0007i   \\ \cline{2-5} 
                       & \multicolumn{1}{l|}{$\mu_{j=3}^{(n=0)}(k=2,\delta=1)$} & 7.8669 + 0.0160i  & \multicolumn{1}{l|}{$\zeta_{n=0,j=3}(k=2, \delta = 1)$} & 0.0173 - 0.0001i   \\ \cline{2-5} 
                       & \multicolumn{1}{l|}{$\mu_{j=4}^{(n=0)}(k=2,\delta=1)$} & 11.0048 + 0.0114i & \multicolumn{1}{l|}{$\zeta_{n=0,j=4}(k=2, \delta = 1)$} & 0.0085 - 0.0000i   \\ \cline{2-5} 
                       & \multicolumn{1}{l|}{$\mu_{j=5}^{(n=0)}(k=2,\delta=1)$} & 14.1443 + 0.0089i & \multicolumn{1}{l|}{$\zeta_{n=0,j=5}(k=2, \delta = 1)$} & 0.0051 - 0.0000i   \\ \hline
$n=1$ & \multicolumn{1}{l|}{$\mu_{j=1}^{(n=1)}(k=2,\delta=1)$} & 2.7440 - 0.0770i  & \multicolumn{1}{l|}{$\zeta_{n=1,j=1}(k=2, \delta = 1)$} & 0.2798 + 0.0336i   \\ \cline{2-5} 
                       & \multicolumn{1}{l|}{$\mu_{j=2}^{(n=1)}(k=2,\delta=1)$} & 6.1160 - 0.0268i  & \multicolumn{1}{l|}{$\zeta_{n=1,j=2}(k=2, \delta = 1)$} & 0.0299 + 0.0003i   \\ \cline{2-5} 
                       & \multicolumn{1}{l|}{$\mu_{j=3}^{(n=1)}(k=2,\delta=1)$} & 9.3161 - 0.0170i  & \multicolumn{1}{l|}{$\zeta_{n=1,j=3}(k=2, \delta = 1)$} & 0.0121 + 0.0000i   \\ \cline{2-5} 
                       & \multicolumn{1}{l|}{$\mu_{j=4}^{(n=1)}(k=2,\delta=1)$} & 12.4855 - 0.0126i & \multicolumn{1}{l|}{$\zeta_{n=1,j=4}(k=2, \delta = 1)$} & 0.0066 + 0.0000i   \\ \cline{2-5} 
                       & \multicolumn{1}{l|}{$\mu_{j=5}^{(n=1)}(k=2,\delta=1)$} & 15.6435 - 0.0100i & \multicolumn{1}{l|}{$\zeta_{n=1,j=5}(k=2, \delta = 1)$} & 0.0042 + 0.0000i   \\ \hline
$n=2$ & \multicolumn{1}{l|}{$\mu_{j=1}^{(n=2)}(k=2,\delta=1)$} & 3.9104 - 0.0072i  & \multicolumn{1}{l|}{$\zeta_{n=2,j=1}(k=2, \delta = 1)$} & 0.0886 + 0.0004i   \\ \cline{2-5} 
                       & \multicolumn{1}{l|}{$\mu_{j=2}^{(n=2)}(k=2,\delta=1)$} & 7.4573 - 0.0026i  & \multicolumn{1}{l|}{$\zeta_{n=2,j=2}(k=2, \delta = 1)$} & 0.0194 + 0.0000i   \\ \cline{2-5} 
                       & \multicolumn{1}{l|}{$\mu_{j=3}^{(n=2)}(k=2,\delta=1)$} & 10.7223 - 0.0017i & \multicolumn{1}{l|}{$\zeta_{n=2,j=3}(k=2, \delta = 1)$} & 0.0090 + 0.0000i   \\ \cline{2-5} 
                       & \multicolumn{1}{l|}{$\mu_{j=4}^{(n=2)}(k=2,\delta=1)$} & 13.9275 - 0.0013i & \multicolumn{1}{l|}{$\zeta_{n=2,j=4}(k=2, \delta = 1)$} & 0.0053 + 0.0000i   \\ \cline{2-5} 
                       & \multicolumn{1}{l|}{$\mu_{j=5}^{(n=2)}(k=2,\delta=1)$} & 17.1084 - 0.0010i & \multicolumn{1}{l|}{$\zeta_{n=2,j=5}(k=2, \delta = 1)$} & 0.0035 + 0.0000i   \\ \hline
\end{tabular}
\caption{By fixing the wave number $k\, =\, 2$, fixing the radius $\delta\, = \,1$ and varying the index $n=0,1,2$.}
\end{table} 
\begin{figure}[H]
    \centering
    \subfigure[A schematic representation of the real part of the eigenfunction $v_{n,1,0}(.,.,.,k=2, \delta = 1)$, for $n=0,1,2$. From the left to the right  $Re\left(v_{0,1,0}(.,.,.,k=2, \delta = 1)\right), Re\left(v_{1,1,0}(.,.,.,k=2, \delta = 1)\right)$ and $Re\left(v_{2,1,0}(.,.,.,k=2, \delta = 1)\right)$.]{
    \centering
    \includegraphics[scale=0.33]{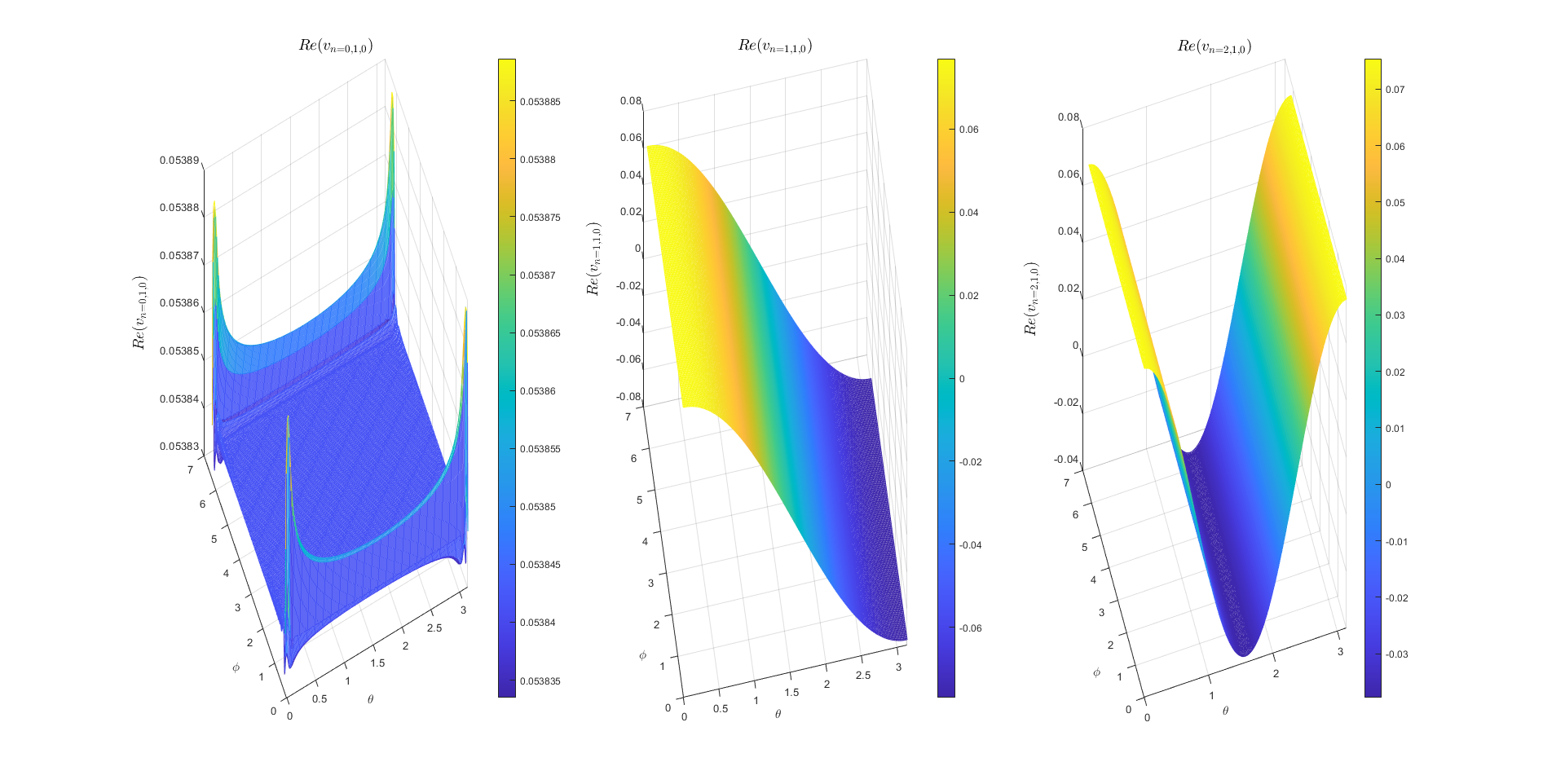} 
    }
    \subfigure[A schematic representation of the imaginary part of the eigenfunction $v_{n,1,0}(.,.,.,k=2, \delta = 1)$, for $n=0,1,2$. From the left to the right  $Im\left(v_{0,1,0}(.,.,.,k=2, \delta = 1)\right), Im\left(v_{1,1,0}(.,.,.,k=2, \delta = 1)\right)$ and $Im\left(v_{2,1,0}(.,.,.,k=2, \delta = 1)\right)$.]{
    \centering
    \includegraphics[scale=0.33]{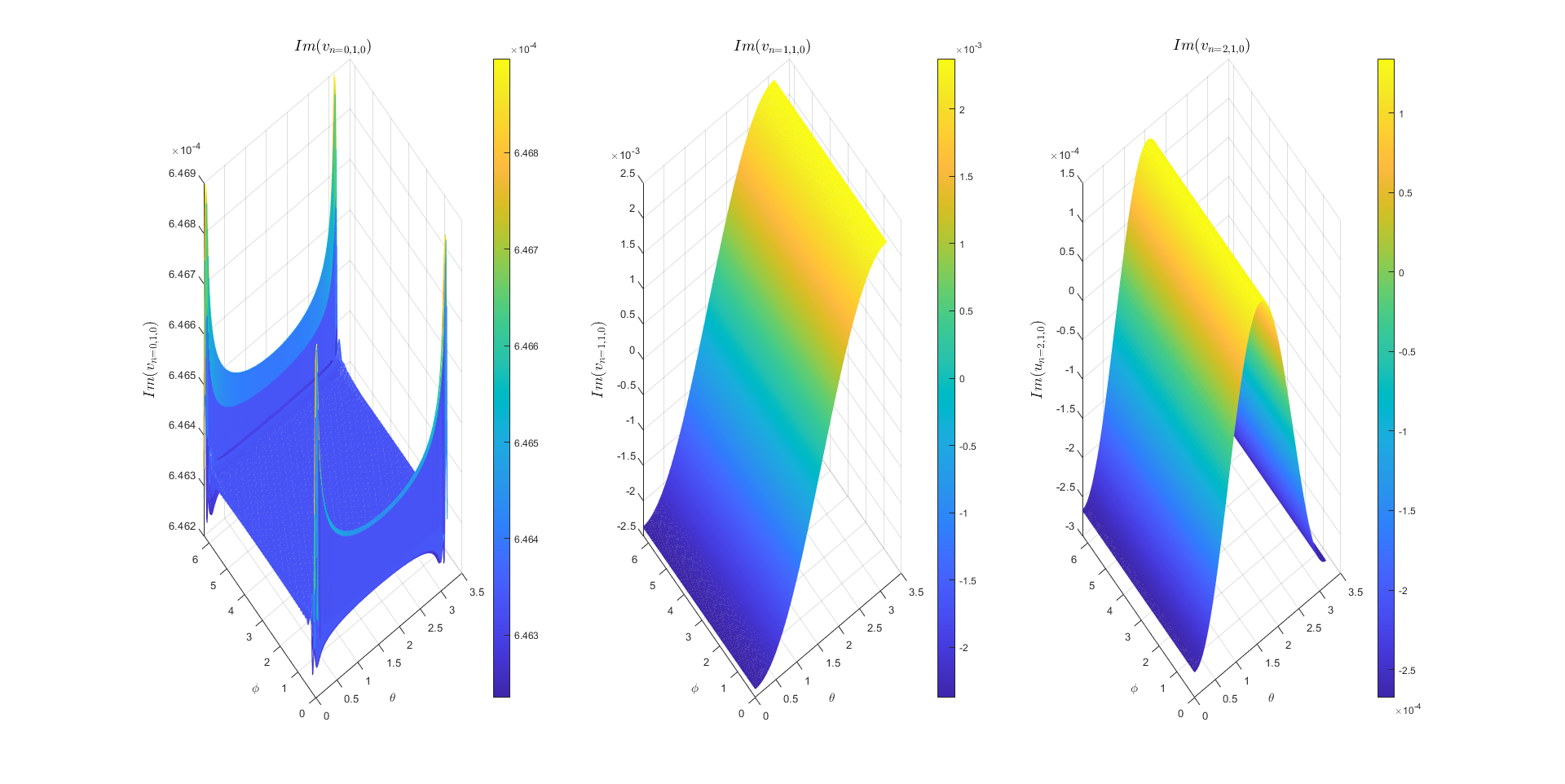}
    }
    \caption{A schematic representation of the eigenfunction $v_{n,1,0}(.,.,.,k=2, \delta = 1)$, for $n=0,1,2$.}
    \label{pic1}
\end{figure}
    \item Eigensystem behavior with respect to the wave number $k$.
\begin{table}[H]
\begin{tabular}{|l|ll|ll|}
\hline
                      & \multicolumn{2}{c|}{Root $\mu_{j}^{(n)}$}                     & \multicolumn{2}{c|}{Eigenvalue $\zeta_{n,j}$}                        \\ \hline
$k=1$  & \multicolumn{1}{l|}{$\mu^{(n=1)}_{j=1}(k=1, \delta=1)$} & 2.7394 - 0.0532i  & \multicolumn{1}{l|}{$\zeta_{n=1,j=1}(k=1, \delta = 1)$} & 0.1535 + 0.0068i  \\ \cline{2-5} 
                      & \multicolumn{1}{l|}{$\mu^{(n=1)}_{j=2}(k=1, \delta=1)$} & 6.1148 - 0.01842i & \multicolumn{1}{l|}{$\zeta_{n=1,j=2}(k=1, \delta = 1)$} & 0.0274 + 0.0001i  \\ \cline{2-5} 
                      & \multicolumn{1}{l|}{$\mu^{(n=1)}_{j=3}(k=1, \delta=1)$} & 9.3154 - 0.01171i & \multicolumn{1}{l|}{$\zeta_{n=1,j=3}(k=1, \delta = 1)$} & 0.0116 + 0.0000i  \\ \cline{2-5} 
                      & \multicolumn{1}{l|}{$\mu^{(n=1)}_{j=4}(k=1, \delta=1)$} & 12.4850 - 0.0086i & \multicolumn{1}{l|}{$\zeta_{n=1,j=4}(k=1, \delta = 1)$} & 0.0064 + 0.0000i \\ \cline{2-5} 
                      & \multicolumn{1}{l|}{$\mu^{(n=1)}_{j=5}(k=1, \delta=1)$} & 15.6431 - 0.0069i & \multicolumn{1}{l|}{$\zeta_{n=1,j=5}(k=1, \delta = 1)$} & 0.0041 + 0.0000i \\ \hline
$k=5$  & \multicolumn{1}{l|}{$\mu^{(n=1)}_{j=1}(k=5, \delta=1)$} & 2.7147 - 0.0249i  & \multicolumn{1}{l|}{$\zeta_{n=1,j=1}(k=5, \delta = 1)$} & -0.0567 + 0.0004i    \\ \cline{2-5} 
                      & \multicolumn{1}{l|}{$\mu^{(n=1)}_{j=2}(k=5, \delta=1)$} & 6.1067 - 0.0085i  & \multicolumn{1}{l|}{$\zeta_{n=1,j=2}(k=5, \delta = 1)$} & 0.0813 + 0.0006i  \\ \cline{2-5} 
                      & \multicolumn{1}{l|}{$\mu^{(n=1)}_{j=3}(k=5, \delta=1)$} & 9.3102 - 0.0054i  & \multicolumn{1}{l|}{$\zeta_{n=1,j=3}(k=5, \delta = 1)$} & 0.0162 + 0.0000i  \\ \cline{2-5} 
                      & \multicolumn{1}{l|}{$\mu^{(n=1)}_{j=4}(k=5, \delta=1)$} & 12.4812 - 0.0040i & \multicolumn{1}{l|}{$\zeta_{n=1,j=4}(k=5, \delta = 1)$} & 0.0076 + 0.0000i \\ \cline{2-5} 
                      & \multicolumn{1}{l|}{$\mu^{(n=1)}_{j=5}(k=5, \delta=1)$} & 15.6401 - 0.0031i & \multicolumn{1}{l|}{$\zeta_{n=1,j=5}(k=5, \delta = 1)$} & 0.0045 + 0.0000i \\ \hline
$k=10$ & \multicolumn{1}{l|}{$\mu^{(n=1)}_{j=1}(k=10, \delta=1)$} & 2.7871 - 0.0424i  & \multicolumn{1}{l|}{$\zeta_{n=1,j=1}(k=10, \delta = 1)$} & -0.0108 + 0.0000i    \\ \cline{2-5} 
                      & \multicolumn{1}{l|}{$\mu^{(n=1)}_{j=2}(k=10, \delta=1)$} & 6.1319 - 0.0153i  & \multicolumn{1}{l|}{$\zeta_{n=1,j=2}(k=10, \delta = 1)$} & -0.0160 + 0.0000i    \\ \cline{2-5} 
                      & \multicolumn{1}{l|}{$\mu^{(n=1)}_{j=3}(k=10, \delta=1)$} & 9.3262 - 0.0098i  & \multicolumn{1}{l|}{$\zeta_{n=1,j=3}(k=10, \delta = 1)$} & -0.0767 + 0.0010i   \\ \cline{2-5} 
                      & \multicolumn{1}{l|}{$\mu^{(n=1)}_{j=4}(k=10, \delta=1)$} & 12.4930 - 0.0072i & \multicolumn{1}{l|}{$\zeta_{n=1,j=4}(k=10, \delta = 1)$} & 0.0178 + 0.0000i  \\ \cline{2-5} 
                      & \multicolumn{1}{l|}{$\mu^{(n=1)}_{j=5}(k=10, \delta=1)$} & 15.6495 - 0.0057i & \multicolumn{1}{l|}{$\zeta_{n=1,j=5}(k=10, \delta = 1)$} & 0.0069 + 0.0000i \\ \hline
\end{tabular}
\caption{By fixing the index $n\, =\, 1$, fixing the radius $\delta\,=\,1$ and varying the wave number $k=1, 5, 10$.}
\end{table} 
\begin{figure}[H]
    \centering
    \subfigure[A schematic representation of the real part of the eigenfunction $v_{1,1,0}(.,.,.,k, \delta = 1)$, for $k=1,5,10$. From the left to the right  $Re\left(v_{1,1,0}(.,.,.,k=1, \delta = 1)\right), Re\left(v_{1,1,0}(.,.,.,k=5, \delta = 1)\right)$ and $Re\left(v_{1,1,0}(.,.,.,k=10, \delta = 1)\right)$.]{
    \centering
    \includegraphics[scale=0.33]{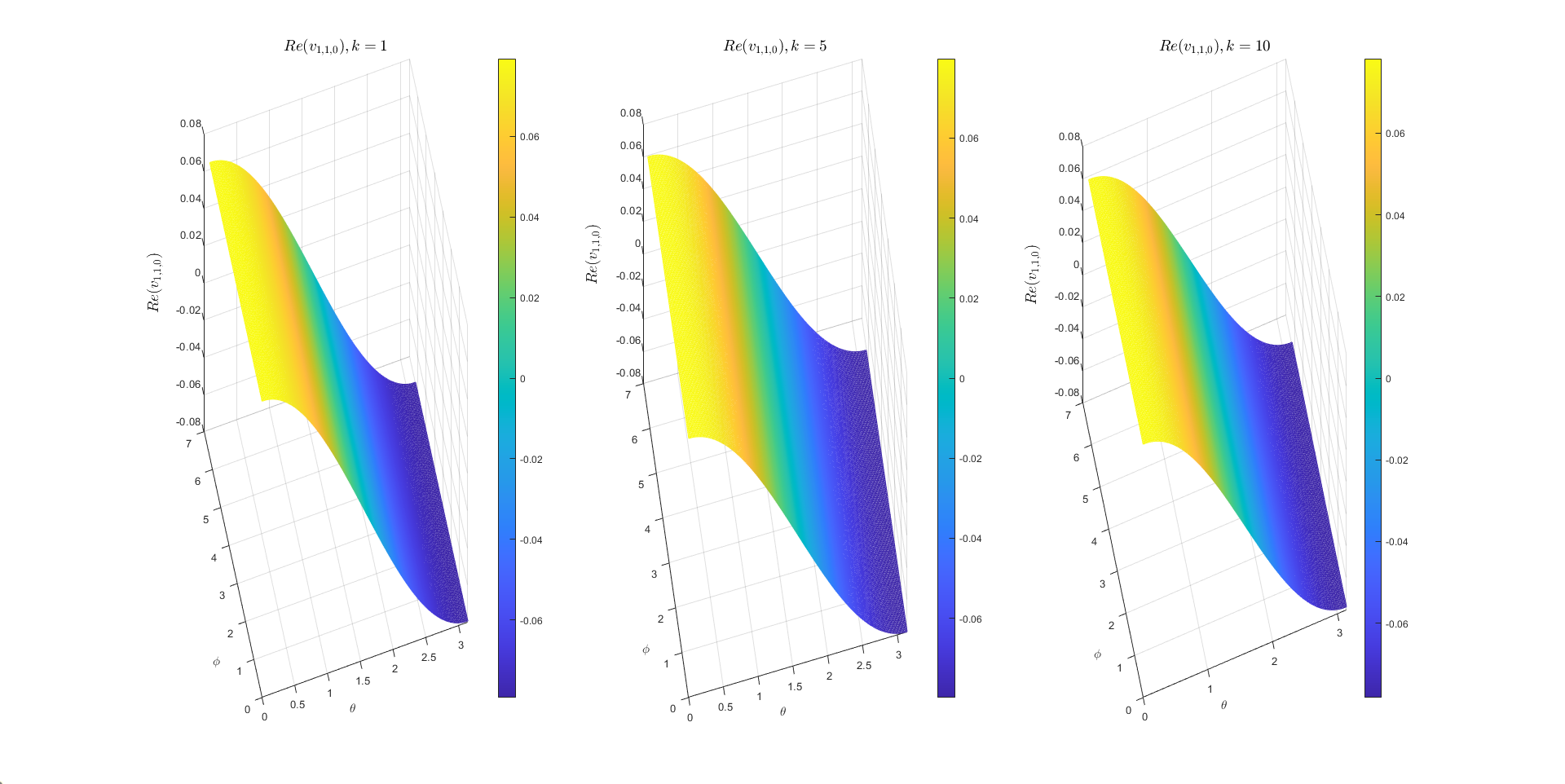}
    }
    \subfigure[A schematic representation of the imaginary part of the eigenfunction $v_{1,1,0}(.,.,.,k, \delta = 1)$, for $k=1,5,10$. From the left to the right  $Im\left(v_{1,1,0}(.,.,.,k=1, \delta = 1)\right), Im\left(v_{1,1,0}(.,.,.,k=5, \delta = 1)\right)$ and $Im\left(v_{1,1,0}(.,.,.,k=10, \delta = 1)\right)$.]{
    \centering
    \includegraphics[scale=0.33]{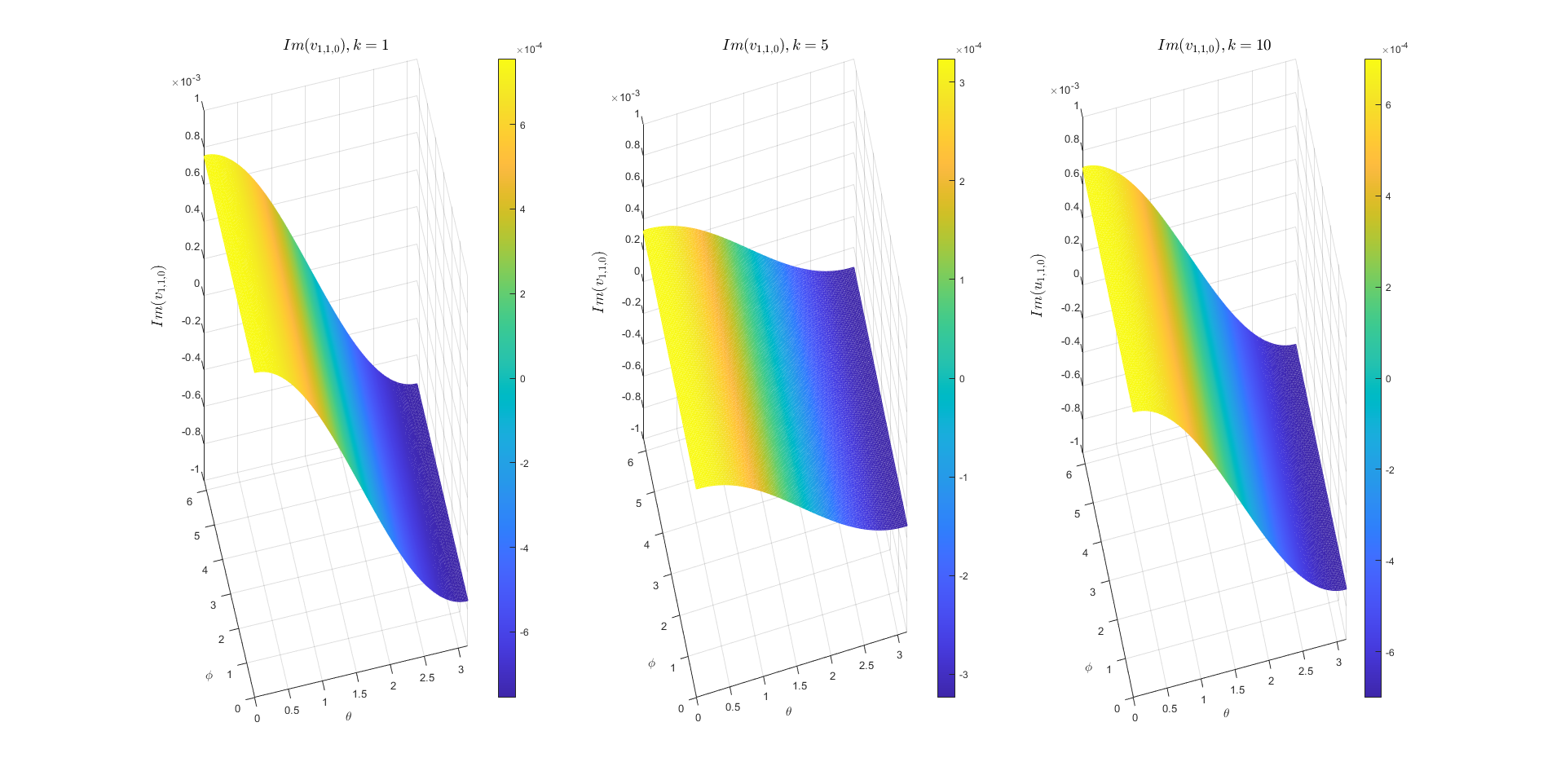}
    }
    \caption{A schematic representation of the eigenfunction $v_{1,1,0}(\cdot, \cdot, \cdot, k, \delta=1)$, for $k=1,5,10 $.}
    \label{pic2}
\end{figure}
    \item[]
    \item Eigensystem behavior with respect to the radius $\delta$.
    \begin{table}[H]
\begin{tabular}{|l|ll|ll|}
\hline
                              & \multicolumn{2}{c|}{Root $\mu^{(n)}_{j}$}                      & \multicolumn{2}{c|}{Eigenvalue $\zeta_{n,j}$}                          \\ \hline
$\delta=0.1$ & \multicolumn{1}{l|}{$\mu^{(n=1)}_{j=1}(k=4, \delta=0.1)$} & 2.1676 - 0.0087i   & \multicolumn{1}{l|}{$\zeta_{n=1,j=1}(k=4, \delta=0.1)$} & 0.0022 + 0.0000i   \\ \cline{2-5} 
                              & \multicolumn{1}{l|}{$\mu^{(n=1)}_{j=2}(k=4, \delta=0.1)$} & 5.9582- 0.0019i    & \multicolumn{1}{l|}{$\zeta_{n=1,j=2}(k=4, \delta=0.1)$} & 0.0002 + 0.0000i  \\ \cline{2-5} 
                              & \multicolumn{1}{l|}{$\mu^{(n=1)}_{j=3}(k=4, \delta=0.1)$} & 9.2170 - 0.0012i   & \multicolumn{1}{l|}{$\zeta_{n=1,j=3}(k=4, \delta=0.1)$} & 0.0001 + 0.0000i  \\ \cline{2-5} 
                              & \multicolumn{1}{l|}{$\mu^{(n=1)}_{j=4}(k=4, \delta=0.1)$} & 12.4126 - 0.0009i  & \multicolumn{1}{l|}{$\zeta_{n=1,j=4}(k=4, \delta=0.1)$} & 0.0000 + 0.0000i \\ \cline{2-5} 
                              & \multicolumn{1}{l|}{$\mu^{(n=1)}_{j=5}(k=4, \delta=0.1)$} & 15.5857 - 0.0007i  & \multicolumn{1}{l|}{$\zeta_{n=1,j=5}(k=4, \delta=0.1)$} & 0.0000 + 0.0000i \\ \hline
$\delta=1$   & \multicolumn{1}{l|}{$\mu^{(n=1)}_{j=1}(k=4, \delta=1)$} & 2.7794 - 0.0069i   & \multicolumn{1}{l|}{$\zeta_{n=1,j=1}(k=4, \delta=1)$} & -0.1208 + 0.0005i       \\ \cline{2-5} 
                              & \multicolumn{1}{l|}{$\mu^{(n=1)}_{j=2}(k=4, \delta=1)$} & 6.1294 - 0.0025i   & \multicolumn{1}{l|}{$\zeta_{n=1,j=2}(k=4, \delta=1)$} & 0.0463 + 0.0000i    \\ \cline{2-5} 
                              & \multicolumn{1}{l|}{$\mu^{(n=1)}_{j=3}(k=4, \delta=1)$} & 9.3246 - 0.0016i   & \multicolumn{1}{l|}{$\zeta_{n=1,j=3}(k=4, \delta=1)$} & 0.0140 + 0.0000i    \\ \cline{2-5} 
                              & \multicolumn{1}{l|}{$\mu^{(n=1)}_{j=4}(k=4, \delta=1)$} & 12.4919 - 0.0012i  & \multicolumn{1}{l|}{$\zeta_{n=1,j=4}(k=4, \delta=1)$} & 0.0071 + 0.0000i   \\ \cline{2-5} 
                              & \multicolumn{1}{l|}{$\mu^{(n=1)}_{j=5}(k=4, \delta=1)$} & 15.6486 - 0.0009i  & \multicolumn{1}{l|}{$\zeta_{n=1,j=5}(k=4, \delta=1)$} & 0.0043 + 0.0000i   \\ \hline
$\delta=10$ & \multicolumn{1}{l|}{$\mu^{(n=1)}_{j=1}(k=4, \delta=10)$} & 4.0363 - 0.03418i  & \multicolumn{1}{l|}{$\zeta_{n=1,j=1}(k=4, \delta=10)$} & -0.0631 + 0.0000i      \\ \cline{2-5} 
                              & \multicolumn{1}{l|}{$\mu^{(n=1)}_{j=2}(k=4, \delta=10)$} & 7.01726 - 0.04365i & \multicolumn{1}{l|}{$\zeta_{n=1,j=2}(k=4, \delta=10)$} & -0.0644 + 0.0000i      \\ \cline{2-5} 
                              & \multicolumn{1}{l|}{$\mu^{(n=1)}_{j=3}(k=4, \delta=10)$} & 10.0187 - 0.0441i  & \multicolumn{1}{l|}{$\zeta_{n=1,j=3}(k=4, \delta=10)$} & -0.0666 + 0.0000i      \\ \cline{2-5} 
                              & \multicolumn{1}{l|}{$\mu^{(n=1)}_{j=4}(k=4, \delta=10)$} & 13.0556 - 0.0410i  & \multicolumn{1}{l|}{$\zeta_{n=1,j=4}(k=4, \delta=10)$} & -0.0699 + 0.0000i      \\ \cline{2-5} 
                              & \multicolumn{1}{l|}{$\mu^{(n=1)}_{j=5}(k=4, \delta=10)$} & 16.1200 - 0.0371i  & \multicolumn{1}{l|}{$\zeta_{n=1,j=5}(k=4, \delta=10)$} & -0.0746 + 0.0000i      \\ \hline
\end{tabular}
\caption{By fixing the index $n\, =\, 1$, the wave number $k\,=\,4$ and varying the radius $\delta = 0.1, 1, 10$.}
\end{table}

\begin{figure}[H]
    \centering
    \subfigure[A schematic representation of the real part of the eigenfunction $v_{1,1,0}(.,.,.,k=4, \delta)$, for $\delta = 0.1,1,10$. From the left to the right  $Re\left(v_{1,1,0}(.,.,.,k=4, \delta = 0.1)\right), Re\left(v_{1,1,0}(.,.,.,k=4, \delta = 1)\right)$ and $Re\left(v_{1,1,0}(.,.,.,k=4, \delta = 10)\right)$.]{
    \centering
    \includegraphics[scale=0.33]{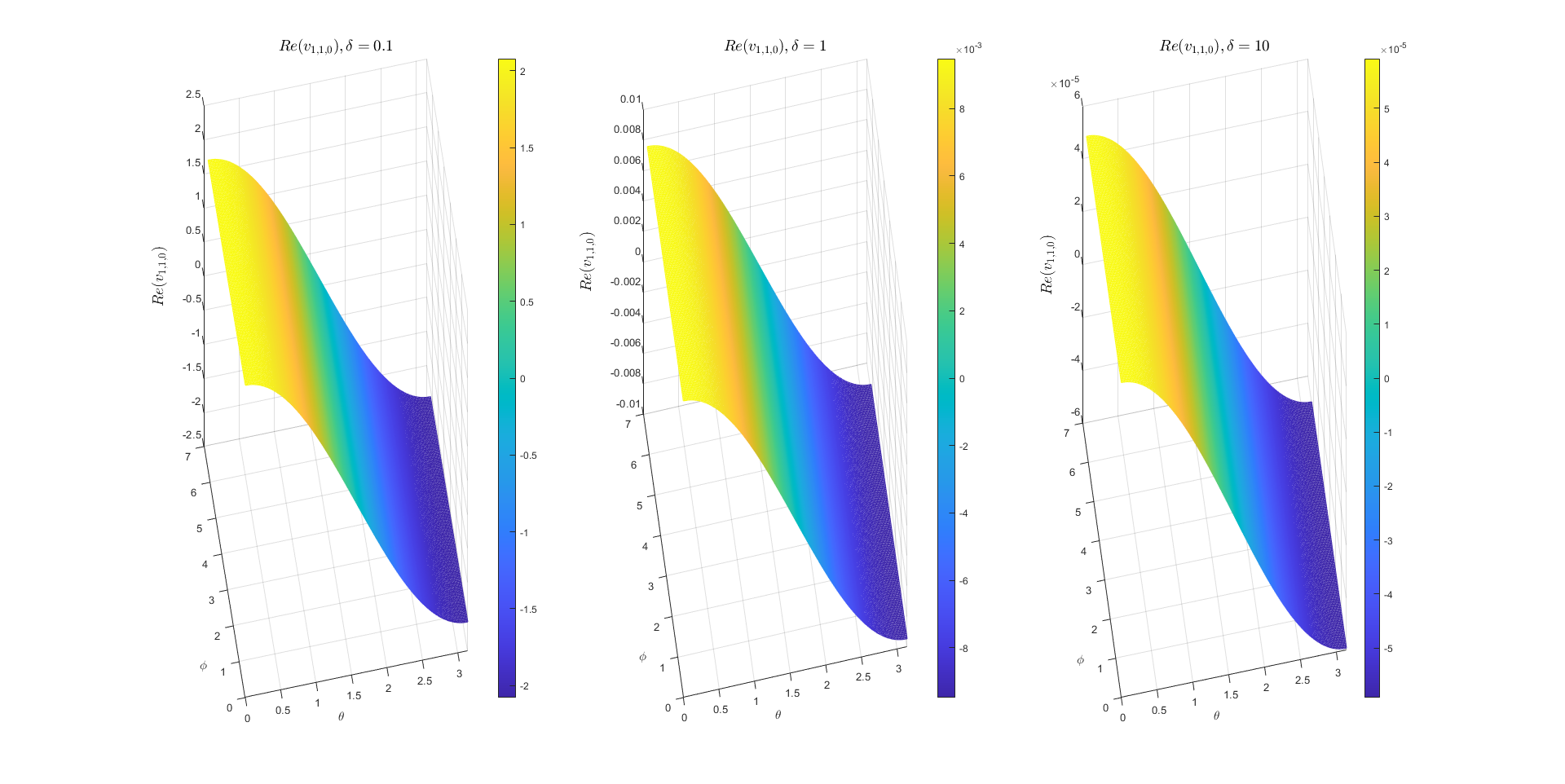}
    }
    \subfigure[A schematic representation of the imaginary part of the eigenfunction $v_{1,1,0}(.,.,.,k=4, \delta)$, for $\delta = 0.1,1,10$. From the left to the right  $Im\left(v_{1,1,0}(.,.,.,k=4, \delta = 0.1)\right), Im\left(v_{1,1,0}(.,.,.,k=4, \delta = 1)\right)$ and $Im\left(v_{1,1,0}(.,.,.,k=4, \delta = 10)\right)$.]{
    \centering
    \includegraphics[scale=0.33]{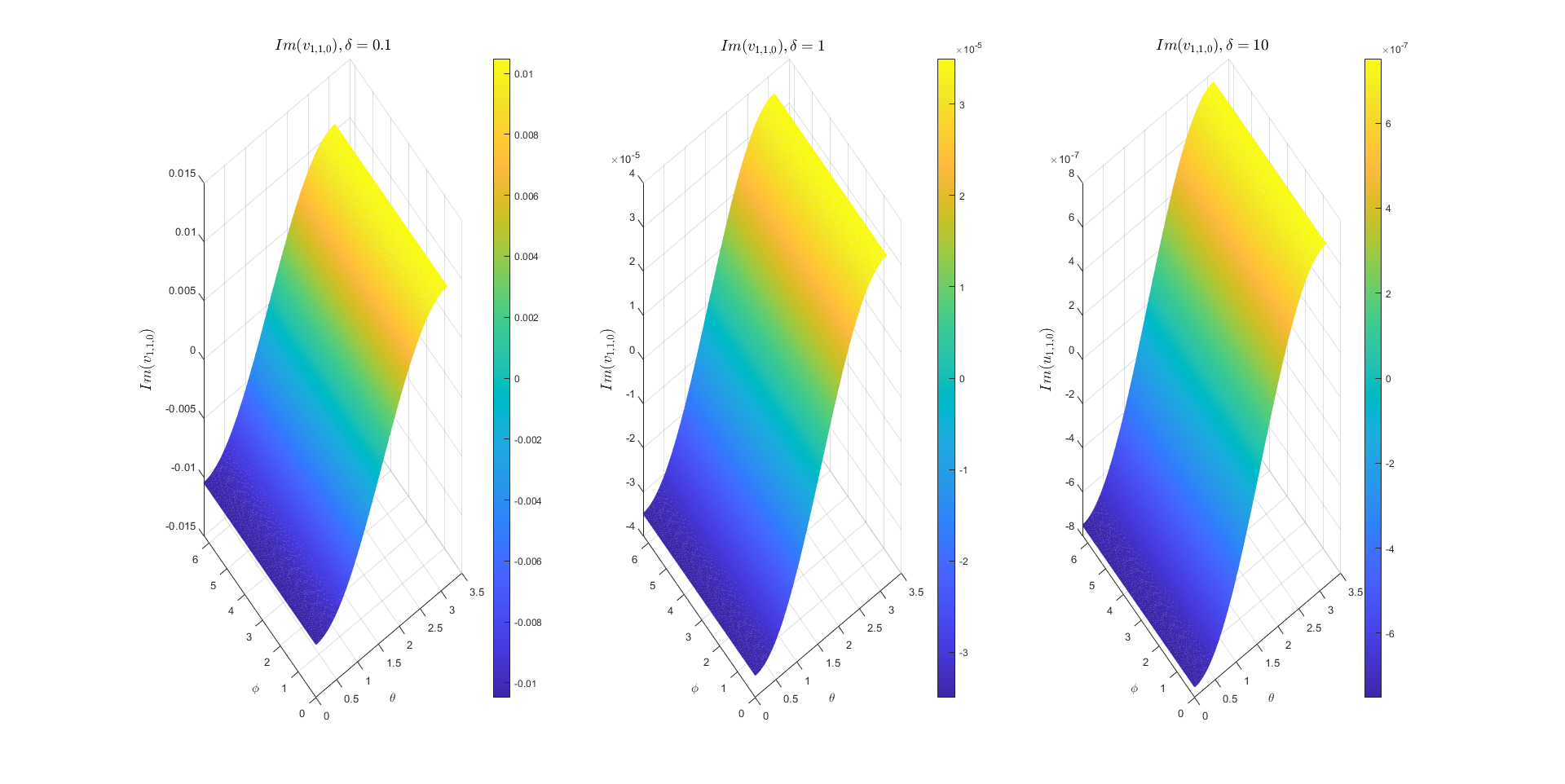}
    }
    \caption{A schematic representation of the eigenfunction $v_{1,1,0}(\cdot, \cdot, \cdot, k=4, \delta)$, for $\delta = 0.1, 1, 10$.}
    \label{pic3}
\end{figure}
%
%
%
\end{enumerate}
The roots  $\mu^{(n)}_{j}(k, \delta)$, for $j, n \in \mathbb{N}$, were been numerically computed using the Matlab function "vpasolve".


\begin{thebibliography}{10}
\bibitem{abramowitz'sbook}
M. Abramowitz and Irene A. Stegun,
\newblock{Handbook of mathematical functions with formulas, graphs, and mathematical tables, US Government printing office, vol. 55, 1968.}
\bibitem{alsenafi2023foldy}
A. Alsenafi, A. Ghandriche and M. Sini,
\newblock{The Foldy-Lax approximation is valid for nearly resonating frequencies, Zeitschrift f{\"u}r angewandte Mathematik und Physik, vol. 74, num. 1, 11, 2023.}
\bibitem{ammari2019subwavelength}
H. Ammari, A. Dabrowski, B. Fitzpatrick, P. Millien and M. Sini,
\newblock{Subwavelength resonant dielectric nanoparticles with high refractive indices, Mathematical Methods in the Applied Sciences, vol. 42, num. 18, 6567--6579, 2019.}
\bibitem{AmmariKang}
H. Ammari and H. Kang,
\newblock{Boundary layer techniques for solving the Helmholtz equation in the presence of small inhomogeneities, Journal of mathematical analysis and applications, vol. 296, num. 1, 190--208, 2004.}
\bibitem{anderson1992spectral}
J. M. Anderson, D. Khavinson and v. Lomonosov,
\newblock{Spectral properties of some integral operators arising in potential theory, The Quarterly Journal of Mathematics, vol. 43, num. 4, 387--407, 1992.}
\bibitem{erdelyi1953}
H. Bateman and A. Erd{\'e}lyi,
\newblock{Higher transcendental functions, volume II, Mc Graw-Hill Book Company, 1953.}
\bibitem{bowman1958}
F. Bowman,
\newblock{Introduction to Bessel functions, Courier Corporation, 1958.}
\bibitem{cartan1945theorie}
H. Cartan,
\newblock{Th{\'e}orie du potentiel newtonien: {\'e}nergie, capacit{\'e}, suites de potentiels, Bulletin de la Soci{\'e}t{\'e} Math{\'e}matique de France, vol. 73, 74--106, 1945.}
\bibitem{chang2008}
T. Chang and K. Lee,
\newblock{Spectral properties of the layer potentials on Lipschitz domains. Illinois Journal of Mathematics, vol. 52, num. 2, 463--472, 2008.}

\bibitem{Liu1}
\newblock{Jinchao Pan, Jijun Liu, Numerical solution of the boundary value problem of elliptic equation by Levi function scheme, to appear in Numer. Methods for PDEs, DOI 10.1002/num.23142.}
\bibitem{colton2019inverse}
D. Colton and R. Kress,
\newblock{Inverse acoustic and electromagnetic scattering theory, 93, 2019, Springer Nature.}
\bibitem{colton2013integral}
D. Colton and R. Kress,
\newblock{Integral equation methods in scattering theory, SIAM, 2013.}
\bibitem{DGS-Bubbles}
A. Dabrowski, A. Ghandriche and M. Sini, 
\newblock{Mathematical analysis of the acoustic imaging modality using bubbles as contrast agents at nearly resonating frequencies, Inverse Problems and Imaging, 2021, 15(3): 555-597.}
\bibitem{EvansBook}
L. C. Evans,
\newblock{Partial differential equations, vol. 19, 
American Mathematical Society, 2022.}
\bibitem{filippi1983integral}
P. J. T. Filippi,
\newblock{Integral equations in acoustics. Theoretical acoustics and numerical techniques, Springer, 1--49, 1983.}
\bibitem{gilbarg2001elliptic}
D. Gilbarg and N. S. Trudinger,
\newblock{Elliptic partial differential equations of second order, 2001, Springer.}
\bibitem{GN}
D. S. Grebenkov and B. T. Nguyen,
\newblock{Geometrical Structure of Laplacian Eigenfunctions, SIAM Review, vol. 55, num. 4, pages 601-667, 2013.}
\bibitem{Suragan}
T. Sh. Kalmenov and D. Suragan,
\newblock{A boundary condition and spectral problems for the Newton potential, Modern aspects of the theory of partial differential equations, pages 187--210, 2011.}
\bibitem{kalmenov2019}
T. S. Kalmenov, M. Ruzhansky and D. Suragan, 
\newblock{On spectral and boundary properties of the volume potential for the Helmholtz equation. Mathematical Modelling of Natural Phenomena, vol. 14, num. 5, 2019.}
\bibitem{kellogg}
O. D. Kellogg,
\newblock{Foundations of potential theory, Springer Science \& Business Media, vol. 31, 2012.}
\bibitem{KS}
M. K. Kerimov and S. L. Skorokhodov,
\newblock{Evaluation of complex zeros of bessel functions $j_{\nu}(z)$  and $i_{n}(z)$ and their derivatives. USSR Computational Mathematics and Mathematical Physics, Vol. 24, Issue 5, 131-141, 1984.}
\bibitem{krasnoselskii1966integral}
M. A. Krasnoselskii, P. P. Zabreiko, E. I. Pustyl'Nik and P. E. Sobolevskii, 
\newblock{Integral operators in spaces of summable functions, Springer, 1966.}
\bibitem{Landau}
J. L. Landau, 
\newblock{Ratios of Bessel Functions and Roots of $\alpha \, J_{\nu}(x) \, + \, x \, J^{\prime}_{\nu}(x) \, = \, 0$. Journal of Mathematical Analysis and Applications, Vol. 240, Issue 1, 174-204, 1999.}
\bibitem{mantile2024point}
A. Mantile and A. Posilicano,
\newblock{The point scatterer approximation for wave dynamics, Partial Differential Equations and Applications, vol. 5, num. 5, 2024.}
\bibitem{miyanishi2017}
Y. Miyanishi and T. Suzuki,
\newblock{Eigenvalues and eigenfunctions of double 
layer potentials. Transactions of the American 
Mathematical Society, vol. 369, num. 11, 8037--8059,
2017.}
\bibitem{nedelec}
J. C. N{\'e}d{\'e}lec,
\newblock{Acoustic and electromagnetic equations: integral representations for harmonic problems, vol. 144, Springer, 2001.}
\bibitem{poincare}
H. Poincare,
\newblock{Theorie du Potential Newtonien, Carre et Naud, 1899; Gabag Reprinted, 1990.}
\bibitem{sommerfeld1943}
A. Sommerfeld,
\newblock{Die ebene und sph{\"a}rische Welle im polydimensionalen Raum, Mathematische Annalen, vol. 119, num. 1, pages 1--20, 1943.}
\bibitem{watson1922}
G. N. Watson,
\newblock{A treatise on the theory of Bessel functions,
  volume 2, The University Press, 1922.}
\bibitem{wilcox2006scattering}
C. H. Wilcox,
\newblock{Scattering theory for the d'Alembert equation in exterior domains, vol. 442, Springer, 2006.}
\end{thebibliography}
\end{document}